\documentclass[12pt]{amsart}
\usepackage[utf8]{inputenc}
\usepackage{amsfonts}
\usepackage{latexsym}
\usepackage{amssymb}
\usepackage{amsmath}
\usepackage{cite}

\usepackage{enumerate}

\usepackage[dvipsnames]{xcolor}

\definecolor{col1}{rgb}{0.6, 0.7, 0.8}
\definecolor{col2}{rgb}{0.7, 0.8, 0.65}
\definecolor{col3}{rgb}{0.8, 0.9, 0.5}
\definecolor{col4}{rgb}{0.91,0.94, 0.53}
\definecolor{col5}{rgb}{0.98,0.99,0.6}

\definecolor{textcol}{rgb}{0.37,0,0.57}
\usepackage{tikz}
\usepackage{forest}
\forestset{
  default preamble={
   for tree={circle, draw, 
            minimum size=2.1em, 
            inner sep=1pt} 
  }
}
 
\usepackage{subfigure}
\usepackage{array}
\newcolumntype{C}[1]{>{\centering\arraybackslash}p{#1}}

\usetikzlibrary{positioning}

\usepackage[left=1.5cm, top=2.5cm,bottom=2.5cm,right=1.5cm]{geometry}

\newcommand{\freeT}{\overrightarrow{\mathcal{T}}}
\newcommand{\freeU}{\overrightarrow{\mathcal{U}}}
\newcommand{\fR}{\overrightarrow{R}}
\newcommand{\fA}{\overrightarrow{\mathbb{A}}}
\newcommand{\fS}{\overrightarrow{\mathbb{S}}}
\newcommand{\fF}{\overrightarrow{\mathbb{F}}}
\newcommand{\fmu}{\overrightarrow{\mu}}
\newcommand{\fT}{\overrightarrow{\mathbb{T}}}

\newcommand{\ind}{{\small 1}\!\!1}

\newtheorem{thm}{Theorem}[section]

\newtheorem{lemma}[thm]{Lemma}

\newtheorem{proposition}[thm]{Proposition}

\newtheorem{conj}[thm]{Conjecture}

{
\theoremstyle{definition}
\newtheorem{example}[thm]{Example}
\newtheorem{rem}[thm]{Remark}
\newtheorem{que}[thm]{Question}
}

\definecolor{gurot}{RGB}{180,20,20}
\definecolor{jorot}{RGB}{220,20,20}

\usepackage{nomencl}
\makenomenclature

\begin{document}


\title[Fa\`a di Bruno's formula and inversion of power series]{Fa\`a di Bruno's formula and inversion of power series}

\author[S.G.G. Johnston]{Samuel G.G. Johnston}
\address{Samuel G.~G.~Johnston: Department of Mathematical Sciences, University of Bath, United Kingdom.}
\email{sggjohnston@gmail.com}

\author[J. Prochno]{Joscha Prochno}
\address{Joscha Prochno: Faculty of Computer Science and Mathematics, University of Passau, Germany.} \email{joscha.prochno@uni-passau.de}




\begin{abstract}
Fa\`a di Bruno's formula gives an expression for the derivatives of the composition of two real-valued functions. In this paper we prove a multivariate and synthesized version of Fa\`a di Bruno's formula in higher dimensions, providing a combinatorial expression for the derivatives of chain compositions $F^{(1)} \circ \ldots \circ F^{(m)}$ of functions $F^{(l)} : \mathbb{R}^N \to \mathbb{R}^N$ in terms of sums over labelled trees. We give several applications of this formula, including a new involution formula for the inversion of multivariate power series. We use this framework to outline a combinatorial approach to studying the invertibility of polynomial mappings, giving a purely combinatorial restatement of the Jacobian conjecture. Our methods extend naturally to the non-commutative case, where we prove a free version of Fa\`a di Bruno's formula for multivariate power series in free indeterminates, and use this formula as a tool for obtaining a new inversion formula for free power series.
\end{abstract}

\maketitle

\section{Introduction} \label{sec:introduction}

\subsection{The one dimensional case}
Before discussing our results in full generality, we give an outline of our approach by first considering the one dimensional case. Let $f,g:\mathbb{R} \to \mathbb{R}$ be smooth functions and consider the derivatives of the product $fg$. The product rule states that $D[fg] = D[f] g + f D[g]$, and by applying an induction argument it is straightforward to prove the more general Leibniz rule, which states that the $n^{\text{th}}$ derivative of $fg$ is given by 
\begin{align} \label{eq:new leibniz}
D^n [ fg ] = \sum_{ S \sqcup T = [n] } D^{ \# S } [f] D^{ \# T} [g],
\end{align}
where $[n] = \{1,\ldots,n\}$, and the sum in \eqref{eq:new leibniz} ranges over all pairs $(S,T)$ of disjoint (possibly-empty) subsets of $[n]$ satisfying $S \cup T = [n]$.

Our somewhat unconventional representation of the Leibniz rule \eqref{eq:new leibniz} as a sum over pairs of subsets is motivated by a desire to avoid using binomial coefficients. This choice is one we will make throughout the paper: wherever possible we will circumvent combinatorial coefficients by indexing every sum through a sufficiently rich collection of combinatorial objects.

Consider now the derivatives of the \emph{composition} $f \circ g$. Here the chain rule $D[f \circ g] = \left( D[f] \circ g \right) D[g]$ plays the role of the product rule, and the analogue of the Leibniz rule \eqref{eq:new leibniz} is furnished by the lesser known Fa\`a di Bruno formula
\begin{align} \label{eq:faa di bruno}
D^n[ f \circ g] = \sum_{ \pi \in \mathcal{P}_{[n]} } \left( D^{ \# \pi} [ f] \circ g  \right) ~ \prod_{ \Gamma \in \pi} D^{ \# \Gamma} [g],
\end{align}
where the sum ranges over $\mathcal{P}_{[n]}$, the collection of set partitions of $[n]$, $\# \pi$ counts the number of blocks in the partition $\pi$, and $\# \Gamma$ counts the number of elements of a block $\Gamma$. 

Suppose now $f(X) = \sum_{ k \geq 1} \frac{f_k}{k!} X^k$ and $g(X) = \sum_{ k \geq 1} \frac{g_k}{k!} X^k$ are formal power series with no constant term. Then Fa\`a di Bruno's formula \eqref{eq:faa di bruno} tells us that the $n^{\text{th}}$ coefficient of the composite power series for $(f \circ g)(X) = \sum_{ k \geq 1} \frac{(f \circ g)_k }{k!} X^k$ is given by 
\begin{align} \label{eq:faa di bruno2}
(f \circ g)_n  = \sum_{ \pi \in \mathcal{P}_{[n]} } f_{ \# \pi}  \prod_{ \Gamma \in \pi} g_{ \# \Gamma}.
\end{align}
We now show that \eqref{eq:faa di bruno2} may be applied to the problem of inverting power series. Indeed, suppose now that $f$ and $g$ are inverses of one another, in the sense that $(f \circ g )(X) = (g \circ f)(X) = X$. Then \eqref{eq:faa di bruno2} endows us with the system of equations
\begin{align} \label{eq:system1}
\sum_{ \pi \in \mathcal{P}_{[n]} } f_{ \# \pi}  \prod_{ \Gamma \in \pi} g_{ \# \Gamma} =
\begin{cases}
1 \qquad & \text{if $n = 1$},\\
0 \qquad & \text{otherwise}.
\end{cases}
\end{align}
Using this system of equations it is straightforward to give an expression for each coefficient $g_k$ in terms of the $(f_k)$ and a collection of rooted trees. To this end, let $\mathbb{S}_k$ be the set of finite rooted graph-theoretic trees $\mathcal{T}$ such that leaves are in bijection with $[k]$ and such that every non-leaf vertex has two or more children. Given such a tree $\mathcal{T}$, and a collection of indeterminates $(h_k : k \geq 2) $, define the $h$-energy of the tree $\mathcal{T}$ by 
\begin{align*}
\mathcal{E}_h( \mathcal{T} ) := \prod_{ v \in I } h_{ \# \text{children of $v$} },
\end{align*}
where the product is taken over $I$, the set of all internal vertices $v$ in the tree (those vertices which are not leaves).


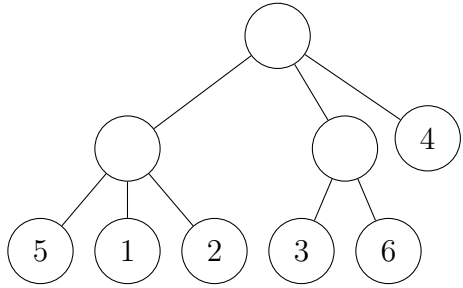
\begin{figure}[ht]
  \centering   
 \begin{forest}
[
[
[5]
[1]
[2]
]
[
[3]
[6]
]
[4]
]
\end{forest}
  \caption{An element $\mathcal{T}$ of $\mathbb{S}_6$. The $h$-energy of this tree is $\mathcal{E}_h(\mathcal{T}) = h_2 h_3^2$.}\label{fig:tree}
\end{figure}
We now use the system of equations \eqref{eq:system1} to give a sketch proof of the following result, which  states that the coefficients of a compositional inverse for a power series with coefficients $(h_k)$ may be written in terms of a sum over a set of trees of $h$-energies $\mathcal{E}_h(\mathcal{T})$. For the sake of simplicity we assume here that $f'(0) = 1$.
\begin{proposition} \label{thm:inversion1}
Let $g(X) = \sum_{ k \geq 0} \frac{ g_k }{k!} X^k$ be the formal inverse of the power series  $f(X) = X - \sum_{ k \geq 2} \frac{ h_k}{k!}  X^k$. Then $g_0 = 0$, $g_1 = 1$, and for $k \geq 2$, the coefficients of $g$ are given by
\begin{align} \label{eq:gform}
g_k = \sum_{ \mathcal{T} \in \mathbb{S}_k } \mathcal{E}_h \left( \mathcal{T} \right).
\end{align}
\end{proposition}
\begin{proof}
It is immediate that $g_0 = 0$, and the fact that $g_1 = 1$ follows from the chain rule. The proof of formula \eqref{eq:gform} for $k \geq 2$ is established by an inductive argument using the system \eqref{eq:system1}. Indeed, plugging $ n = k+1$ in \eqref{eq:system1}, using the fact that $f_1 = 1$ and $f_j = - h_j$ for $j \geq 2$, we obtain
\begin{align*}
g_{k+1} =  \sum_{ \pi \in \mathcal{P}_{k+1 } : \# \pi  \geq 2 } h_{ \# \pi } \prod_{ \Gamma \in \pi} g_{ \# \Gamma}.
\end{align*}
Using the inductive hypothesis to expand each term $g_{ \# \Gamma}$ in the product as a tree, by thinking of each term in this product as a subtree of a tree in which the root has $\# \pi$ children,  each term in the sum on the right-hand side now corresponds to $\mathcal{E}_h(\mathcal{T})$ for a tree $\mathcal{T}$ in $\mathbb{S}^{k+1}$. 
\end{proof}

While the contents of the discussion above are well known, with Proposition \ref{thm:inversion1} following as a fairly straightforward consequence of the single-variable Lagrange inversion formula, the purpose of the present article is to expand this discussion to multivariate functions and power series in both commutative and non-commutative variables, and supply formulas for the composition and inversion of these functions. We now overview both of these cases.

\subsection{The commutative case} The main commutative result in this article is Theorem \ref{thm:main}, a generalisation of Fa\`a di Bruno's formula \eqref{eq:faa di bruno} giving a combinatorial expression for the higher derivatives 
\begin{align} \label{eq:chain}
D^\alpha \left[ F^{(1)} \circ \ldots \circ F^{(m)} \right],
\end{align}
where $F^{(l)}:\mathbb{R}^N \to \mathbb{R}^N$ are smooth functions and $\alpha$ is an arbitrary multi-index. The expression for \eqref{eq:chain} is given in terms of a sum over labelled trees with generations $0$ through $m$: the number of leaves in each of these trees is equal to degree of the multi-index $\alpha$, and every leaf lies in generation $m$. Figure \ref{fig:treefinalex} gives an example of one of these trees.

\begin{figure}[ht!]
  \centering   
 \begin{forest}
[2, fill=col2, name=g01,
[1,  fill=col1, name=g11,
[1, fill=col1, name=g21,
[{2,1}, fill=col2, name=g31,]
[{3,1},fill=col3, name=g32,]
]
]
[3,fill=col3, name=g12,
[1,fill=col1, name=g22,
[{3,2},  fill=col3, name=g33,]
]
[5,fill=col5, name=g23,
[{1,1},  fill=col1, name=g34,]
]
]
[4, fill=col4, name=g13,
[4, fill=col4, name=g24,
[{1,2},  fill=col1, name=g35,]
]
]
]
   \node (a) [right=of g01 -| g24.east] {Generation $0$};
    \node (b) [right=of g13,] {Generation $1$};
   \node (c) [right=of g24,]       {Generation $2$};
   \node (d) [right=of g35,]       {Generation $3$};
    \draw[dashed,] (g01) -- (a) ;
   \draw[dashed,] (g13) -- (b);
   \draw[dashed,] (g24) -- (c);
   \draw[dashed,] (g35) -- (d);
\draw[dashed,] (g11) -- (g12);
\draw[dashed,] (g12) -- (g13);
\draw[dashed,] (g21) -- (g22);
\draw[dashed,] (g22) -- (g23);
\draw[dashed,] (g23) -- (g24);
\draw[dashed,] (g34) -- (g35);
\draw[dashed,] (g33) -- (g34);
\draw[dashed,] (g32) -- (g33);
\draw[dashed,] (g31) -- (g32);
\end{forest}
  \caption{The fifth order partial derivatives of a three-fold composition $F \circ G \circ H$ are given in terms of a sum over labelled trees with $3+1$ generations and $5$ leaves. }\label{fig:treefinalex}
\end{figure}
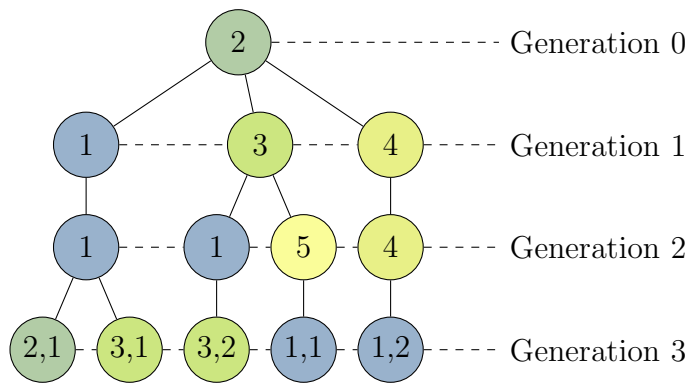

We present several combinatorial applications of Theorem \ref{thm:main}, including most importantly a new inversion formula for formal power series. Indeed, suppose that $F = (F_1,\ldots,F_N)$ is an $N$-dimensional formal power series with $i^{\text{th}}$ component
\begin{align*}
F_i(X) = \sum_{ \alpha \in \mathbb{Z}_{ \geq 0}^N } \frac{F_{i,\alpha}}{ \alpha!} X^\alpha,
\end{align*}
where $F_{i,\alpha}$ take values in some ring. In Theorem \ref{thm:general inversion} we show that the coefficients of the compositional formal inverse power series $G = (G_1,\ldots,G_N)$ may be given in terms of a sum over labelled trees of energies depending on the coefficients $(F_{i,\alpha})$ --- provided the linear term of $F$ is an invertible matrix. The special case of our result where the linear term of $F$ is identical appears implicitly in Haiman and Schmitt \cite{HS}. Roughly speaking, we find that when the linear term is non-identical, it interlaces the energy function of the tree. In Section \ref{sec:JC} we conclude our discussion of commutative power series by applying our framework to outline a combinatorial path to proving the Jacobian conjecture, building heavily on the work of Wright \cite{wright1,wright2,wright3,wright4} and Singer \cite{singer1,singer2,singer3}.

\subsection{The non-commutative case}
We then turn to looking at the non-commutative case, where the methods above are naturally adapted to studying composition and compositional inversion of formal power series in free variables $X_1,\ldots,X_N$ which are not assumed to commute (we use the term `free' throughout to mean `non-commutative'). In essence, our results in the commutative case extend to the free case via the rule of thumb \emph{``free variables means planar trees"}. 

Our first result in this direction is a free version of the multi-dimensional Fa\`a di Bruno formula, suitable for understanding compositions of formal expressions such as
\begin{align*}
f( X_1 , X_2 , X_3)  = X_1 X_2 X_1 + X_2 X_3 + 4 X_3 X_2,
\end{align*}
where $X_i$ are free variables. We permit such expressions to have infinitely many terms, and call such objects \emph{free formal power series}.

In Theorem \ref{thm:free fdb} we find that if $F^{(1)} ,\ldots,F^{(m)}$ are free formal power series in $N$ variables such that each $F^{(l)}$ has $N$ components $(F^{(l)}_1,\ldots,F^{(l)}_N)$, then the coefficient of a term $X_{i_1} \ldots X_{i_k}$ may be given in terms of a sum over labelled \emph{planar} trees with $m$ generations and $k$ leaves. By a \emph{planar} tree, we refer to a rooted tree in which the children of every internal vertex are ranked from left to right.

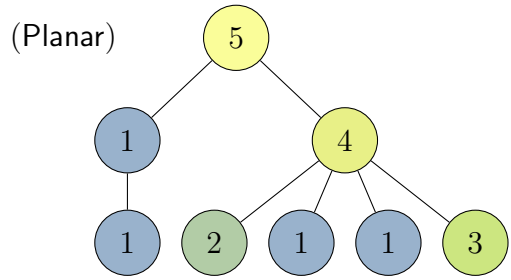
\begin{figure}[ht!]
  \centering   
 \begin{forest}
[5, fill=col5, name=L1,
[1,  fill=col1
[1, fill=col1, name=A]
]
[4,fill=col4
[2,  fill=col2]
[1,  fill=col1]
[1,fill=col1]
[3,  fill=col3, name=B]
]
]
\node (a) [left=of L1] {($\mathsf{Planar}$)};
\end{forest}
  \caption{The coefficient of $X_1 X_2 X_1 X_1 X_3$ in the fifth component of the composition $F \circ G$ may be given in terms of a sum over planar trees with two generations and leaves of types $1,2,1,1,3$ from left to right.}\label{fig:quick free tree}
\end{figure}

Consider now the compositional inversion of free power series. Our final result, Theorem \ref{thm:free inversion}, states that like the commutative case, a mapping $F$ in non-commutative variables has a left-inverse $G$ if and only if its linear term is invertible as matrix with entries in the coefficient ring, and that in this case, the mapping $G$ is also a right-inverse. Furthermore, Theorem \ref{thm:free inversion} gives an explicit formula for each inverse series $G_i$ in terms of rooted planar trees.

\subsection{Structure of the paper}

The remainder of the paper is structured as follows:
\begin{itemize}
\item In Section \ref{sec:results} we give full statements of our main results in the commutative case. This amounts to our multivariate generalisation of Fa\`a di Bruno's formula, Theorem \ref{thm:main}, our inversion formulas for multivariate power series, Theorems \ref{thm:inversion} and \ref{thm:general inversion}, and a new combinatorial statement for the Jacobian conjecture, Conjecture \ref{JC3}. 

\item In Section \ref{sec:free} we give statements of our results in the non-commutative case. These results consist of Theorem \ref{thm:free fdb} --- our free Fa\`a di Bruno formula for multivariate power series in $N$ non-commuting indeterminates --- as well as an inversion formula for such series,  Theorem \ref{thm:free inversion}.

\item  In Section \ref{sec:applications} we collect together several applications of Fa\`a di Bruno's formula and the power series inversion formulas, including a representation for the Hermite polynomials in terms of partitions, as well as a new formula for the reciprocal of a power series. Finally, we give a brief overview of recent appearances of Fa\`a di Bruno's formula in the genealogical structure of a class of random trees known as Galton-Watson trees.

\item The remaining sections are dedicated to giving proofs of the results stated in Sections \ref{sec:results} and \ref{sec:free}: in Section \ref{sec:fdbproof} we prove our generalisation of the classical Fa\`a di Bruno formula, Theorem \ref{thm:main}, in Section \ref{sec:lagrangeproof} we prove our power series inversion formulas, Theorem \ref{thm:inversion} and Theorem \ref{thm:general inversion}, and in the final section, Section \ref{sec:free proofs}, we prove Theorems \ref{thm:free fdb} and \ref{thm:free inversion} which give statements about the composition and inversion of power series in non-commuting variables.
\end{itemize}

\section{The commutative case} \label{sec:results}
We shall now present the main results and contributions of this article in the commutative case. In order to begin discussing the higher dimensional setting, we first need to introduce some additional notation.

\subsection{Multi-indices and the higher dimensional Leibniz rule}
Recall that $[N] := \{1,\ldots,N\}$. Whenever $S$ is a finite subset of $[N] \times \mathbb{Z}_{\geq 0}$, let $\# S$ be the multi-index whose $i^{\text{th}}$ component counts the number of elements of the form $(i,a)$ in $S$. Now given a multi-index $\alpha \in \mathbb{Z}_{ \geq 0}^N$, define
\begin{align*}
[\alpha] := \left \{ (i, a ) \in \mathbb{Z}^2 : i \in [N], 1 \leq a \leq \alpha_ i \right\} \subseteq [N] \times \mathbb{Z}_{ \geq 0}.
\end{align*}
 Plainly $\# [\alpha] = \alpha$. Finally, let $\left( \mathbf{e}_i \right)_{ i \in [N] }$ be the standard basis for $\mathbb{Z}_{ \geq 0}^N$.

We are now equipped to provide a high-dimensional version of the Leibniz rule, namely for functions $f^{(1)}, f^{(2)}, \ldots ,f^{(m)}:\mathbb{R}^N \to \mathbb{R}$ the higher derivatives of their product $f^{(1)} f^{(2)} \ldots f^{(m)}$ are given by 
\begin{align} \label{eq:leib2}
D^\alpha \left[ f^{(1)}  f^{(2)} \ldots f^{(m)}   \right] = \sum_{ S_1 \sqcup \ldots \sqcup S_m = [\alpha] } \prod_{ l = 1}^m D^{ \# S_i } [f^{(i)} ],
\end{align}
where the sum ranges over all $m$-tuples $(S_1,\ldots,S_m)$ of disjoint (and possibly empty) subsets of $[\alpha]$ whose union is equal to $\alpha$. The expression \eqref{eq:leib2} is easily proved by induction.

With a view to stating Theorem \ref{thm:main}, which gives an expression analogous to \eqref{eq:leib2} for higher derivatives $D^\alpha \left[ F^{(1)} \circ \ldots \circ F^{(m)} \right]$ of chain compositions, in the next section we introduced labelled rooted trees.

\subsection{Labelled rooted trees} \label{sec:trees}
 We say a finite graph $(V,E)$ is a rooted tree if it is a finite tree with a designated root vertex $v_0$. For $k \in \mathbb{Z}_{ \geq 0}$, let $V_k$ denote the set of vertices whose graph distance from the root is $k$. We refer to $V_k$ as generation $k$, and note that $V_0 = \{ v_0 \}$. Each vertex $w$ in $V_{k+1}$ has a unique parent vertex $v$ in $V_{k}$, and in this case we say $w$ is a child of $v$. We say a vertex is a leaf if it has no children, and we say it is internal if it has at least one child. The vertex set of a rooted tree has a decomposition into a disjoint union $V = L \sqcup I$, where $L$ are the leaves of the tree and the set $I$ consists of the internal vertices of the tree. 

For $i \in [N]$ and a finite subset $S$ of $[N] \times \mathbb{Z}_{ \geq 0}$ let $\mathbb{T}_{i,S}$ to be the set of labelled trees $\mathcal{T}$ with root type $i$ and leaf types $S$. More specifically, $\mathbb{T}_{i,S}$ is the set of quadruplets $\mathcal{T} := (V,E,\tau,\phi)$ with the following properties:
\begin{itemize}
\item The pair $(V,E)$ denotes a rooted tree.
\item The labelling function $\phi:L \to S$ is a bijection between the leaves of the tree $L$ and the set $S$, giving each  each leaf a \emph{label} in $S$.
\item The typing function $\tau:V \to [N]$ gives each vertex of the tree a \emph{type} in $[N]$ according to certain rules: it gives the root type $i$, and the type given to each leaf respects the label of that leaf, in the sense that that whenever $\phi(v) = (j,a)$ for some $a \in \{1,2,\ldots\}$, we have $\tau(v) = j$. There are no constraints on the types of non-root internal vertices.
\end{itemize}

\begin{figure}[ht!]
  \centering   
 \begin{forest}
[1, fill=col1
[1,  fill=col1
[{4,3}, fill=col4]
[{3,4},fill=col3]
]
[3,fill=col3
[{3,1},  fill=col3]
[{1,1},  fill=col1]
]
[4, fill=col4
[{1,2},  fill=col1]
]
]
\end{forest}
  \caption{An element $\mathbb{F}_{i,S}(2)$ where $i = 1$ and $S = \{ (1,1), (1,2), (3,1),(3,4),(4,3)\}$. A vertex labelled $(k,a)$ has type $k$. The colours of the vertices correspond to their types. Note that every leaf lies in the final generation.}\label{fig:treefinal}
\end{figure}
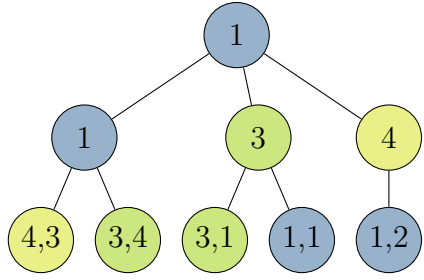

We say two labelled trees  $(V,E,\tau,\phi)$ and $(V',E',\tau',\phi')$ are \emph{isomorphic} if there is a bijection $\Psi:V \to V'$ between the underlying vertex sets preserving all of the structure of the tree: $\left( \Psi(u), \Psi(v) \right) \in E'$ if and only if $(u,v) \in E$, $\tau'(\Psi(u)) = \tau(u)$ and $\phi'(\Psi(u) ) = \phi(u)$. Whenever we speak of a set or collection of trees, technically we mean a set or collection of isomorphism classes according to this equivalence.

So in summary, for subsets $S$ of $[N] \times \{1,2,\ldots \}$, $\mathbb{T}_{i,S}$ is the collection of (isomorphism classes of) labelled trees with root type $i$ and whose leaves are in bijection with $S$. We will consider several subsets of $\mathbb{T}_{i,S}$:
\begin{itemize}
\item The set of \emph{final} trees of length $m$, $\mathbb{F}_{i,S}(m)$: those trees in $\mathbb{T}_{i,S}$ such that every leaf is contained in generation $m$. The tree in Figure \ref{fig:treefinal} is a final tree of length $2$.
\item The set of \emph{proper} trees $\mathbb{S}_{i,S}$: those trees in $\mathbb{T}_{i,S}$ such that every internal vertex has two or more children.
\item The set of \emph{alternating} trees $\mathbb{A}_{i,S}$: the subset of $\mathbb{T}_{i,S}$ consisting of trees such that every vertex in some $V_{2k}$ has exactly one child, and every vertex in some $V_{2k+1}$ has no children or two or more children. In particular, every leaf of an alternating tree lies in some $V_{2k+1}$. See e.g. Figure \ref{fig:treealternating}
\end{itemize}

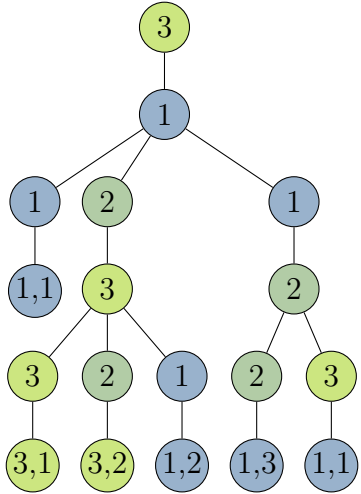
\begin{figure}[ht!]
  \centering   
 \begin{forest}
   for tree={circle, draw, 
            minimum size=1.61em, 
            inner sep=0.7pt} 
[3, fill=col3 
[1, fill=col1
[1, fill=col1
[{1,1},  fill=col1] 
]
[2, fill=col2
[3, fill=col3
[3, fill=col3
[{3,1}, fill=col3]
]
[2,fill=col2
[{3,2}, fill=col3]
]
[1,  fill=col1
[{1,2},  fill=col1]
]
]
] 
[1,  fill=col1
[2, fill=col2
[2, fill=col2
[{1,3}, fill=col1]
]
[3, fill=col3
[{1,1}, fill=col1]
]
]
]
]
]
\end{forest}
  \caption{An alternating tree in $A_{i,S}$ with root type $i = 3$, and leaf labelling set $S = \{ (1,1),(1,2),(1,3), (3,1),(3,2)\} = [3 \mathbf{e}_1 + 2 \mathbf{e}_3]$. Note that every vertex in an even generation has exactly one child.}\label{fig:treealternating}
\end{figure}

To lighten notation we write $\mathbb{T}_{i,\alpha} := \mathbb{T}_{i,[\alpha]}$, and we do similarly with $\mathbb{F}_{i,\alpha}(m)$, $\mathbb{S}_{i,\alpha}$ and $\mathbb{A}_{i,\alpha}$. Final trees $\mathbb{F}_{i,\alpha}(m)$ will appear in our generalisation of Fa\`a di Bruno's formula, specifically for the $\alpha^{\text{th}}$ derivative of the $i^{\text{th}}$ component of $F^{(1)} \circ \ldots \circ F^{(m)}$. The proper trees $\mathbb{S}_{i,\alpha}$ will appear in our inversion formula for power series with identity linear terms. The unusual looking alternating trees $\mathbb{A}_{i,\alpha}$ are used in our inversion formula for power series with non-identical linear terms.

For an internal vertex $v$ the outdegree $\mu(v)$ of $v$ is defined to be the multi-index whose $i^{\text{th}}$ component is given by 
\begin{align*}
\mu(v)_i := \# \{ w \in V : \text{$w$ is a child of $v$ and $\tau(w) = i$} \}.
\end{align*}

\begin{figure}[ht!]
  \centering   
 \begin{forest}
[1, fill=col1
[{3,1}, fill=col3]
[4, fill=col4]
[4, fill=col4]
]
\end{forest}
  \caption{An internal vertex $v$ with type $\tau(v) = 1$ and outdegree $\mu(v) = \mathbf{e}_3 + 2 \mathbf{e}_4$. }\label{fig:v}
\end{figure}
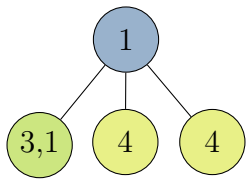

All of our formulas in the commutative case for the composition and inversion of functions take the form $\sum_{ \mathcal{T} \in \mathcal{C} } \mathcal{E} (\mathcal{T})$, where $\mathcal{C}$ is some collection of labelled trees, and the energy has the form $\mathcal{E}(\mathcal{T}) = \prod_{v \in I} K ( \tau(v), \mu(v) )$ for some function $K$ depending on the type of the vertex and its outdegree.

We are now equipped to discuss the specific cases in the next two sections.

\subsection{Generalisation of Fa\`a di Bruno's formula}
Let $F^* := (F^{(1)},\ldots,F^{(m)})$ be an $m$-tuple of smooth functions from $\mathbb{R}^N$ to $\mathbb{R}^N$. 
We now give a formula for the higher derivatives of their composition in terms of the derivatives of of each individual $F^{(l)}$ and a sum over a collection of trees. 

Whenever $F:\mathbb{R}^N \to \mathbb{R}^N$, let $D_i^\alpha[F]:\mathbb{R}^N \to \mathbb{R}$ denote the $\alpha^{\text{th}}$ derivative of the $i^{\text{th}}$ component. Given a labelled tree $\mathcal{T} \in \mathbb{F}_{i,\alpha}(m)$ define the energy function $\mathcal{E}_{F^*}(\mathcal{T}): \mathbb{R}^N \to \mathbb{R}$ by 
\begin{align*}
\mathcal{E}_{F^*}(\mathcal{T}) := \prod_{ l = 1}^{m} \prod_{ v \in V_{l-1} } D^{ \mu(v) }_{ \tau(v)} [F^{(l)} ]  \circ F^{(l+1)} \circ \ldots \circ F^{(m)}.
\end{align*}
Suppose now $F^* = (F.G)$, and consider the tree $\mathcal{T}$ in $\mathbb{F}_{5,\alpha}(2)$ in Figure \ref{fig:treefinal}. For this tree we have 
\begin{align*}
\mathcal{E}_{F^*}(\mathcal{T} ) = \left( D_{5}^{ \mathbf{e}_1 + \mathbf{e}_3 +  \mathbf{e}_4 } [F] \circ G \right) D_1^{\mathbf{e}_3 + \mathbf{e}_4} [G] D_3^{\mathbf{e}_1 + \mathbf{e}_3} [G] D_4^{\mathbf{e}_1}[G] .
\end{align*}
We are now equipped to state our main result, Theorem \ref{thm:main}, which states that the higher derivatives of chain compositions may be given in terms of a sum over tree energies.
\begin{thm} \label{thm:main}
The function $D_i^\alpha[ F^{(1)} \circ \ldots \circ F^{(m)} ]:\mathbb{R}^N \to \mathbb{R}$ is given by 
\begin{align*}
D_i^\alpha [ F^{(1)} \circ \ldots \circ F^{(m)} ](s) = \sum_{ \mathcal{T} \in \mathbb{F}_{i,\alpha}(m) } \mathcal{E}_{F^*}(\mathcal{T})(s).
\end{align*}
\end{thm}

\begin{rem} \label{rem:main}
A formula for the case where $F^{(l)} : \mathbb{R}^{N_l} \to \mathbb{R}^{N_{l+1}}$ are maps between Euclidean spaces of different dimensions is easily obtained from Theorem \ref{thm:main} by means of an embedding argument. Indeed, choose $N$ large enough so that $N \geq N_l$ for all $l$, and then canonically associate each $F^{(l)}$ with $\hat{F}^{(l)} : \mathbb{R}^N \to \mathbb{R}^N$, where $\hat{F}^{(l)}$ is independent of the final $N-N_l$ variables, and the final $N-N_{l+1}$ of $\hat{F}^{(l)}$ components are zero. Applying Theorem \ref{thm:main} to the chain $\hat{F}^{(1)} \circ \ldots \circ \hat{F}^{(m)}$, this procedure tells us that
\begin{align*}
D_i^\alpha \left[ F^{(1)} \circ \ldots \circ F^{(m)} \right] (s) = \sum_{ \mathcal{T} \in \mathbb{F}_{i,\alpha}(m) } \mathcal{E}_{F^*}(\mathcal{T})(s),
\end{align*}
where the sum is supported only on trees where the type function $\tau$ is restricted so that for each $\ell$, the vertices in the $\ell^{\text{th}}$ generation $V_l$ may only have types in $\{1,\ldots,N_\ell\} \subset [N]$. 
\end{rem}

\vspace{8mm}

Theorem \ref{thm:main} reduces significantly when each $F^{(l)}$ is assumed to be a power series with zero constant term and we evaluate the result at the origin. Namely suppose the $i^{\text{th}}$ component of each $F^{(l)}$ is given by
\begin{align*}
F^{(l)}_i(X_1,\ldots,X_N) = \sum_{ \alpha \neq 0} \frac{ F^{(l)}_{i,\alpha}}{ \alpha!} X^\alpha.
\end{align*}
Then by setting $X = 0 $ in Theorem \ref{thm:main} we see that the coefficient of $X^\alpha$ in the $i^{\text{th}}$ component of the composition power series $F^{(1)} \circ \ldots \circ F^{(m)}$ is given by $\frac{1}{\alpha!} (F^{(1)} \circ \ldots \circ F^{(m)} )_{i, \alpha}$, where
\begin{align} \label{eq:fdbpower}
 (F^{(1)} \circ \ldots \circ F^{(m)} )_{i, \alpha} := \sum_{ \mathcal{T} \in \mathbb{F}_{i,\alpha}(m) } \prod_{ l = 1}^{m} \prod_{v \in V_{l-1}} F^{(l)}_{\tau(v),\mu(v)}.
\end{align}
We now take a brief look at the special case $m=2$, for which some results have appeared in the literature. Setting $ m=2$ in \eqref{eq:fdbpower} we obtain
\begin{align} \label{eq:fdbpower2}
(F \circ G)_{i, \alpha} = \sum_{ \mathcal{T} \in \mathbb{F}_{i,\alpha}(2)} F_{i, \mu(v_0) } \prod_{ v \in V_1} G_{ \tau(v) , \mu(v) }.
\end{align}
The formula \eqref{eq:fdbpower2} essentially appears in Haiman and Schmitt \cite{HS}, who give their formula in terms of what they call a \emph{multicolor partition}. Taking a moment to sketch this idea here, each tree $\mathbb{F}_{i,\alpha}(2)$ of length $2$ induces a partition $\pi$ of $[\alpha]$, by letting the blocks of $\pi$ correspond to the vertices in the middle generation $V_1$: we say two elements $(i,j)$ and $(i',j')$ of $[\alpha]$ are in the same block of the partition if the leaves with these labels are children of the same vertex in $V_1$. The types of the vertices in $V_1$ then induces a function $\tau:\pi \to [N]$. Appealing to this correspondence we may rewrite \eqref{eq:fdbpower2} as
\begin{align*} 
(F \circ G)_{i, \alpha } = \sum_{ \pi \in \mathcal{P}_{[\alpha]}} \sum_{ \tau: \pi \to [N]} F_{i,\# \tau} \prod_{ \Gamma \in \pi} G_{ \tau(\Gamma), \# \Gamma},
\end{align*}
where the former sum ranges over all partitions $\pi$ of $[\alpha]$, the latter sum ranges over all functions $\tau$ labelling the blocks of $\pi$, and $\# \tau$ denotes the multi-index whose $i^{\text{th}}$ counts the number of blocks of $\pi$ such that $\tau(\Gamma) = i$..

Before discussing inversion of power series in the next section, we close this section with a brief discussion of the Fa\`a di Bruno formula and its generalisations. While the formula, is as the name indicates, typically attributed to Italian mathematician Fa\`a di Bruno \cite{Brunoa, Brunob, Bruno1}, it first appeared in the calculus book \cite{Arbogast} of Arbogast in the year 1800. 
We refer the reader to Johnson \cite{johnson} for a more thorough historical discussion of Fa\`a di Bruno formula.

Several intermediate generalisations of Fa\`a di Bruno's formula have appeared in the literature. Abraham and Robbin \cite[Section 1.4]{AR} develop a composite mapping formula for sufficiently differentiable functions $f:E \to F$ and $g:F \to G$ defined on Banach spaces, only giving the integer coefficients implicitly through a recurrence relation. Their composite mapping formula is mainly used in \cite{AR} as a tool to prove the Glaeser rough composition theorem. Constantine and Savits \cite{CS} study the case where $f:\mathbb{R}^b \to \mathbb{R}^c$ and $g: \mathbb{R}^a \to \mathbb{R}^b$, giving the derivatives for $D^\alpha [ f \circ g ]$ in terms of various combinatorial coefficients. The book \emph{Vertex Operator Algebras and the Monster} by Frenkel, Lepowsky and Meurman involves a formal calculus including a Faà di Bruno formula for derivations.  There are countless related formulas out there, for instance in \cite{AA1,gzyl,LR,EM,hardy,yang}.

\subsection{Inversion of power series} \label{sec:inversion}
In this section we turn to looking at combinatorial inversion of formal power series. We broaden our framework, working with the ring $R := \mathbb{K}[[X_1,\ldots,X_N]]$ of formal power series in $N$ determinates with coefficients in a commutative ring $\mathbb{K}$.

Suppose now that we have a vector $F = (F_1,\ldots,F_N)$ in $R^N$: namely each $F_i$ is an element of $R$. Each such $F$ may be thought of as a collection $F = \left( F_{i,\alpha} : i \in [N] , \alpha \in \mathbb{Z}_{ \geq 0}^N \right)$ of elements of the ring $\mathbb{K}$. The Fa\`a di Bruno formula for power series \eqref{eq:fdbpower} gives us an associative composition sending a pair $(F,G)$ of elements in $R^N$ to $F \circ G$ in $R^N$. Moreover, the element $I = \left( I_{i,\alpha} : i \in [N], \alpha \in \mathbb{Z}_{ \geq 0}^N \right)$ of $R^N$ defined by $I_{i,\alpha} = 1$ if and only if $\alpha = \mathbf{e}_i$ is an identity for the composition in the sense that $F \circ I = I \circ F = F$. We are interested in the problem of classifying which elements $F$ of $R^N$ are invertible, that is in identifying for which $F$ there exists a $G$ such that $F \circ G = G \circ F = I$, and in obtaining combinatorial expressions for the coefficients of the inverse $G$ in terms of the coefficients of $F$. 

Given an element $F$ of $R^N$, we write $J(F)$ for its Jacobian matrix, the $N \times N$ matrix given by 
\begin{align*} 
J(F) := \left( \frac{\partial F_i}{\partial X_j}  \right)_{1 \leq i,j \leq N} = \left( \sum_{ \alpha \in \mathbb{Z}_{ \geq 0}^N } \frac{ F_{i, \alpha + \mathbf{e}_j} }{ \alpha!} X^\alpha \right)_{1 \leq i,j \leq N}.
\end{align*}
The Jacobian matrix $J(F)$ is a matrix whose entries take values in the ring $R = \mathbb{K}[[X_1,\ldots,X_N]]$. Evaluating the Jacobian matrix at $X = 0$, we obtain the \emph{linear term} of the power series $F$, the matrix
\begin{align*}
J(F)(0) := \left( F_{i,\mathbf{e}_j} \right)_{1 \leq i,j \leq N}.
\end{align*}
The linear term $J(F)(0)$ is a matrix whose entries take values in $\mathbb{K}$. We will see below that a power series $F$ has a compositional inverse $G$ if and only if its linear term $J(F)(0)$ is invertible as an $N \times N$ matrix with coefficients in the commutative ring $\mathbb{K}$. 

Classifying the coefficients of the power series inverse is more delicate, and we begin with the known case where $F(0) = 0$ and $J(F)(0)$ is the identity matrix, sketching how \eqref{eq:fdbpower2} may be used to identify the derivatives (or equivalently, the coefficients of the power series for) $G$, much in analogy to how we developed the one-dimensional inversion formula in the introduction

To this end, we introduce some notation. We will consider inverting power series in the following two subsets of $R^N$:
\begin{itemize}
\item Let $R_0^N$ denote the set of power series in $R^N$ such that $F(0) = 0$. 
\item Let $R_1^N$ denote the set of power series in $R^N$ such that $F(0) = 0$ and $J(F)(0) = I$. 
\end{itemize}
We begin with the simpler case, where $F$ is an element of $R_1^N$. The $i^{\text{th}}$ component of such an $F$ has the form
\begin{align*}
F_i(X)  = X_i + \sum_{ |\alpha| \geq 2 } \frac{ F_{i,\alpha}}{\alpha!} X^\alpha.
\end{align*}
Now suppose $G$ is an inverse of $F$. Since $F \circ G$ is the identity, we must have $(F \circ G)_{i,\mathbf{e}_i } = 1$ and $(F \circ G)_{i, \alpha } = 0$ whenever $\alpha \neq \mathbf{e}_i$. Using this observation in conjunction with \eqref{eq:fdbpower2} we obtain the relations
\begin{align} \label{eq:system}
\sum_{ \mathcal{T} \in \mathbb{F}_{i,\alpha}(2)} F_{i, \mu(v_0) } \prod_{ v \in V_1} G_{ \tau(v) , \mu(v) } =
\begin{cases}
1 &\text{if $\alpha = \mathbf{e}_i$},\\
0 &\text{otherwise},
\end{cases}
\end{align}
which hold for all $i \in [N]$ and all multi-indices $\alpha$. The system of equations \eqref{eq:system} gives us a straightforward inductive proof of Theorem \ref{thm:inversion}, which states that each $G_{i,\alpha}$ may be written as a sum of weighted trees whose weights are given in terms of the $\left(  F_{ i , \alpha } \right)$. 

Given a proper labelled tree $\mathcal{T} \in \mathbb{S}_{i,\alpha}$ and a collection $H := \left( H_{i,\alpha}  : 1 \leq i \leq N, \alpha \in \mathbb{Z}_{\geq 0}^N : |\alpha| \geq 2 \right)$ of elements of $\mathbb{K}$, we define the $H$-energy of the tree $\mathcal{T}$ to be the element of $\mathbb{K}$ given by 
\begin{align} \label{eq:energy}
\mathcal{E}_H \left( \mathcal{T}  \right) := \prod_{ v \in \mathcal{I} }H_{ \tau(v) , \mu(v) } 
\end{align}

The following theorem states that the coefficients of inverse power series may be given in terms of sums of tree energies over proper trees.

\begin{thm}\label{thm:inversion}
Let $F$ be an element of $R_1^N$, so that the $i^{\text{th}}$ component of $F$ takes the form
\begin{align*}
F_i(X) = X_i - \sum_{ |\alpha| \geq 2 } \frac{ H_{i,\alpha}}{\alpha!} X^\alpha
\end{align*}
for some coefficients $(H_{i,\alpha})$. Then the $i^{\text{th}}$-component of the inverse $G$ of $F$ is given by $G_i(X) = X_i +   \sum_{ |\alpha| \geq 2 } \frac{ G_{i,\alpha}}{\alpha!} X^\alpha$, where for $|\alpha| \geq 2$,
\begin{align} \label{eq:GDEF}
G_{i, \alpha}  = \sum_{ \mathcal{T} \in \mathbb{S}_{i,\alpha} } \mathcal{E}_H \left( \mathcal{T}  \right).
\end{align}

\end{thm}
As mentioned in the introduction, Theorem \ref{thm:inversion} is implicit in Corollary 2 of Haiman and Schmitt \cite{HS}. In Section \ref{sec:fdbproof}, we give two proofs of Theorem \ref{thm:inversion}, both of which rely on Theorem \ref{thm:main} as a key tool. 
\vspace{5mm}

With a view to discussing the Jacobian conjecture below, we would like to highlight a key property of formal inversion of power series. Since every $F \in R_1^N$ may be written
\begin{align*}
F_i(X) = X_i - \sum_{ |\alpha| \geq 2} \frac{H_{i,\alpha}}{\alpha!} X^\alpha,
\end{align*}
we may canonically associate $R_1^N$ with the set
\begin{align*}
R_1^N := \left( H_{i,\alpha} : i \in [N], \alpha \in \mathbb{Z}_{\geq 0}^N : |\alpha| \geq 2 \right).
\end{align*}
(The choice to work with minus the nonlinear coefficients simplifies several formulas we encounter below.) Under this association, consider the map $\Phi: R_1^N \to R_1^N$ defined by setting
\begin{align*}
\Phi(H)_{i, \alpha} := - \sum_{ \mathcal{T} \in \mathbb{S}_{ i ,\alpha} } \mathcal{E}_H(\mathcal{T}).
\end{align*}
The following remark is a consequence of the fact that in the setting of Theorem \ref{thm:inversion}, the coefficients of $F$ may be recovered from the coefficients of the inverse $G$. We will use the mapping $\Phi$ below in our discussion of the Jacobian conjecture.

\begin{rem} \label{rem:involution}
The mapping $\Phi$ is an involution on $R_1^N$. That is, $\Phi \circ \Phi$ is the identity map on $R_1^N$. 
\end{rem}

\vspace{5mm}

Next we consider the inversion of power series with non-identity linear terms, showing that the inverse coefficients may be given in terms of a sum of alternating tree energies over alternating trees.

Indeed, given collections
\begin{align*}
\left ( H_{i,\alpha} : i \in [N],  |\alpha| \geq 2 \right) \qquad \text{and } \qquad \left( Q_{i, \mathbf{e}_j} : i,j \in [N] \right)
\end{align*}
of elements of a ring $\mathbb{K}$, and an alternating tree $\mathcal{T} \in \mathbb{A}_{i,\alpha}$, define the $(H,Q)$ energy of $\mathcal{T}$ to be the double product
\begin{align*}
\mathcal{E}^{\mathsf{alt}}_{H,Q} \left( \mathcal{T} \right) = \prod_{ \text{$v$  even} } Q_{ \tau(v), \mu (v) } \prod_{ \text{$v$   odd} } H_{ \tau(v) , \mu (v) }
\end{align*}
where the first product is taken over internal vertices whose graph distance from the root is even, and the latter over internal vertices for which it is odd. 

\begin{thm}\label{thm:general inversion} 
Let $F$ be an element of $R_0^N$, so that the $i^{\text{th}}$ compenent of $F$ takes the form 
\begin{align*}
F_i(X) = \sum_{ j = 1}^N P_{i,\mathbf{e}_j} X_j - \sum_{ |\alpha| \geq 2 } \frac{ H_{i,\alpha}}{\alpha!} X^\alpha
\end{align*}
for some matrix $P = \left( P_{i, \mathbf{e}_j } \right)_{ 1 \leq i,j \leq N}$ and some coefficients $(H_{i,\alpha})$. 

If $P$ is invertible, then $F$ has a compositional power series inverse $G$. Moreover, the $i^{\text{th}}$ component of $G$ is given by 
\begin{align*}
G_i(s) = \sum_{ j = 1}^N Q_{i, \mathbf{e}_j } X_j + \sum_{ |\alpha| \geq 2} \frac{ G_{i,\alpha}}{ \alpha!} X^\alpha, 
\end{align*} 
where $Q$ is the matrix inverse of $P$, and for $|\alpha| \geq 2$, the coefficients $G_{i,\alpha}$ are given by 
\begin{align} \label{eq:gsum}
G_{i,\alpha}  := \sum_{ \mathcal{T} \in \mathbb{A}_{i,\alpha} } \mathcal{E}^{\mathsf{alt}}_{H,Q}\left( \mathcal{T} \right).
\end{align}
\end{thm}

We remark that thanks to Theorem \ref{thm:general inversion}, the mapping $\Phi$ defined above Remark \ref{rem:involution} has an extension to the set of all invertible power series which is still an involution. 

Moreover, it is straightforward to see that the more general result, Theorem \ref{thm:general inversion} implies the special case Theorem \ref{thm:inversion}. Indeed, suppose that $P$ in Theorem \ref{thm:general inversion} is the identity matrix. Then the inverse $Q$ of $P$  is also the identity, and it follows that each sum in \eqref{eq:gsum} is supported only on trees such that every internal even vertex has outdegree $\mu(v) = \mathbf{e}_{\tau(v)}$, i.e. all vertices in odd generations have the same type as their parent, and in this case we have $\prod_{ \text{$v$  even} } Q_{\tau(v), \mu(v)} = 1$. By collapsing each edge from an even parent to their odd child, we obtain a canonical bijection between those trees in $\mathbb{A}_{i,\alpha}$ such that every odd vertex has the same type as their parent, and trees in $\mathbb{S}_{i,\alpha}$. That this bijection is energy preserving follows from the fact that $\prod_{ \text{$v$  even} } Q_{\tau(v), \mu(v)} = 1$.

Let us take a moment to discuss other formulas for inversion of power series. As mentioned, Theorem \ref{thm:inversion} appears in implicitly in the work of Haiman and Schmitt, where it is used to express the algebra of power series with composition in terms of the powerful \emph{incidence algebra} framework. A less explicit version of Theorem \ref{thm:general inversion} appears in Cheng et al. \cite{cheng}.

Bass, Connell and Wright \cite{BCW} provide inversion formulas for the entire inverse series $G$ in terms of a sum of tree energies over a collection of trees of arbitrary size. (See also the more general Wright \cite{wright1}.) Their analogues for tree energies are themselves polynomials, taking the form \begin{align*}
\mathcal{E}^{\mathsf{BCW}}(\mathcal{T})= \prod_{ v \in I} H_{\tau(v),\mu(v)} \prod_{v \in L } x_{\tau(v)},
\end{align*} 
where the leaves are included as monomials in the energy products. There are also a large collection of inversion formulas variously referred to as Lagrange inversion formulas, see Gessel \cite{gessel} for a comprehensive discussion.
Finally, let us mention that tree formalisms have been used to tackle problems in countless other related areas. See for instance Carletti \cite{carletti}, who studies a problem in dynamical systems, as well as work by Gentile and coauthors \cite{gentile, BG}.

\subsection{The Jacobian conjecture} \label{sec:JC}
We now discuss the celebrated \emph{Jacobian conjecture}.  Recall that $R^N = \mathbb{K} [[ X_1,\ldots,X_N]]^N$ is the set of $N$-tuples $(F_1,\ldots,F_N)$ of formal power series in $N$ variables with coefficients in a ring $\mathbb{K}$. We say an element $F = \left(F_{i,\alpha} : i \in [N], \alpha \in \mathbb{Z}_{ \geq 0}^N \right)$ of $R^N$ is a \emph{polynomial mapping} if only finitely many of the values $(F_{i,\alpha})$ are non-zero. A polynomial mapping is said to be a \emph{polynomial automorphism} if there exists another polynomial mapping $G$ such that $F \circ G = G \circ F = I$. The Jacobian conjecture is concerned with identifying which polynomial mappings are polynomial automorphisms, and has been the subject of a large deal of research over the last fifty years. We refer the reader to the book \cite{VDEbook} for the most comprehensive source on the Jacobian conjecture.

For the sake of concreteness, we work over the complex numbers, setting $\mathbb{K} = \mathbb{C}$, though it is known \cite[Proposition 1.12]{VDEbook} that the statement over $\mathbb{C}$ is equivalent to statements over general fields of characteristic zero. The Jacobian conjecture asserts that in order for a polynomial mapping $F$ to be a polynomial automorphism, it is sufficient to check the global invertibility of its Jacobian matrix:

\begin{conj}[The Jacobian conjecture] \label{JC} Let $N \geq 1$ be any integer. 
Let $F$ be a polynomial mapping such that the Jacobian determinant $j(F) : \mathbb{C}^N \to \mathbb{C}$ defined by 
\begin{align*}
j(F) := \det \left( J(F) \right) = \det_{i,j = 1}^N \left( \frac{ \partial F_i }{ \partial X_j } \right) 
\end{align*}
is equal in value on $\mathbb{C}^N$ to a constant in $\mathbb{C} - \{0\}$. Then $F$ is a polynomial automorphism.
\end{conj}

Several partial results exist for the Jacobian conjecture, usually studying the invertibility of polynomial mappings in a certain dimension, of a certain degree, or with certain other structural properites. Most notably, building on work by Moh \cite{moh}, Wang \cite{lcwang} has shown that the Jacobian conjecture is true for all quadratic mappings.

A large amount of literature on the Jacobian conjecture has been concerned with reducing the complexity of the problem. A foundational idea in this direction is the following reduction due to 
Bass, Connell and Wright \cite{BCW}.

\begin{thm}[Bass, Connell and Wright \cite{BCW}] \label{thm:BCW}
In order to prove Conjecture \ref{JC} it is sufficient to fix any $\delta \geq 3$, and consider maps of the form $F = I - H$, where $H$ is a homogenous polynomial of degree $\delta$. Moreover, we may assume the Jacobian matrix of $H$ is nilpotent.
\end{thm}

We now make a few remarks about Theorem \ref{thm:BCW}. It is straightforward to derive a weaker version of Theorem \ref{thm:BCW}, namely that in order to prove the Jacobian conjecture one need only consider maps of the form $F = I - H$, where $H$ contains only degree two terms or higher. To see this, first note that we may assume $F(0) = 0$, since translations are clearly polynomial automorphisms. Now suppose $F$ is a polynomial mapping fixing the origin and let $L$ be the (invertible) linear transformation $L := J(F)(0)$ associated with evaluating the Jacobian matrix of $F$ at zero. Then the map $\widetilde{F} := F \circ L^{-1}$ has the desired form $\widetilde{F} = I - H$, and $\widetilde{F}$ is a polynomial automorphism if and only if $F$ is a polynomial automorphism. 

The fact that we may further assume that $H$ is homogenous of some degree greater than or equal to three is less elementary, and uses stabilisation ideas from K-theory, which involves treating maps of high degrees in low dimensions as lower degree maps in higher dimensions \cite[Section 2]{BCW}.

When $H$ is homogenous, it is fairly straightforward to show that $j(F)$ constant implies that $H$ is nilpotent. Indeed, if $F = I - H$ has constant determinant, then evaluating this determinant at the origin we see that $j(F)(X) =1$ for all $X \in \mathbb{C}^N$. Define the dilation map $\varphi_t : \mathbb{C}^N \to \mathbb{C}^N$ by $\varphi_t( X ) := t X$ and set $F_t := \varphi_t^{-1} \circ F \circ \varphi_t$. If $F = I -H$ where $H$ is homogenous of degree $\delta \geq 3$, then 
\begin{align*}
F_t = I - t^{\delta-1} H.
\end{align*}
Since $\varphi_t$ is a polynomial automorphism, $F_t$ is a polynomial automorphism if and only if $F$ is. Moreover, by the chain rule $j(F_t)(X) = j(F)(X) = 1$ for all $X \in \mathbb{C}^N$. In particular, for all $t > 0$,
\begin{align} 
1 &= \det \left( I - t^{\delta-1} J(H) \right) \nonumber \\
&= \exp \left( \mathrm{Tr} \log \left( I - t^{\delta-1} J(H) \right) \right) \nonumber\\
&= 1 + \sum_{ k \geq 1} a_k t^{k(\delta-1)} \mathrm{Tr} \left( J(H)^k \right), \label{eq:tracey}
\end{align}
for some constants $(a_k)_{k \geq 1}$. For \eqref{eq:tracey} to hold for every $t$, we must have that every power of $J(H)$ is traceless. In particular, every eigenvalue of $J(H)$ is zero, so that $J(H)$ is nilpotent. It follows that whenever $F = X - H$ has constant Jacobian determinant and $H$ is nilpotent, $J(H)$ is nilpotent.

There are several further reductions to Theorem \ref{thm:BCW}. Dru\.zkowski \cite{druzkowski} showed that we may assume further that each component of $H$ is the cube of a linear map (such mappings are known as Dru\.zkowski mappings). De Bondt and Van den Essen \cite{DBVDE} showed that we may assume that the Jacobian matrix of $H$ is symmetric. Remarkably, if both of these properties are assumed simultaneously for a mapping --- i.e. a mapping $F$ is both symmetric and the cube of a linear map --- then this mapping is a polynomial automorphism \cite{DBVDE2}. 

Several authors, most notably Wright \cite{wright1,wright2,wright3,wright4}, Singer \cite{singer1,singer2,singer3}, Zeilberger \cite{zeilberger} and Abdesselam \cite{AA2}, have remarked on approaching the Jacobian conjecture from a combinatorial standpoint. We would like to collect together some of their ideas here in our notation to give a purely combinatorial formulation of the Jacobian conjecture that we hope will motivate future research. 

To this end, recall the mapping $\Phi$ defined above Remark \ref{rem:involution}, which takes an element of $R_1^N$, i.e. a collection of coefficients $H := \left( H_{i,\alpha } : i \in [N] , \alpha \in \mathbb{Z}_{ \geq 0}^N \right)$ and outputs a second collection of coefficients $\Phi(H) := \left( \Phi(H)_{i,\alpha}: i \in [N] , \alpha \in \mathbb{Z}_{ \geq 0}^N \right)$ in such a way that formal power series with components
\begin{align*}
F_i = X_i - \sum_{ |\alpha| \geq 2 } \frac{ H_{i,\alpha}}{\alpha!} X^\alpha \qquad \text{and} \qquad G_i = X_i - \sum_{ |\alpha| \geq 2} \frac{ \Phi(H)_{i,\alpha}}{\alpha!} X^\alpha
\end{align*}
are compositional inverses of one another. (In particular, $\Phi$ is an involution.)

In order to pursue this direction further, it is useful to phrase the nilpotency of the non-linear term $H$ in combinatorial terms. We say the Jacobian matrix of $H$ has index of nilpotency $m$ if $J(H)^m$ is equal to zero on $\mathbb{C}^N$, and $m$ is the smallest integer with this property. Whenever  $J(H)$ is an $N \times N$ nilpotent matrix, the index of nilpotency must be at most $N$. 

If $H$ has index of nilpotency $m$, then certain energy sums over a type of trees we call \emph{ferns} are equal to zero. A fern of length $m$ is a tree with a designated path of vertices $\{v_0,v_1,\ldots,v_m\}$ starting from the root such that each $v_{i+1}$ is a child of $v_i$, and every vertex not equal to some $v_i$ is a leaf. We call $\{v_0,\ldots,v_m\}$ the spine of the fern.

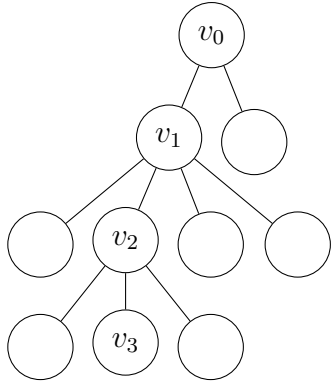
\begin{figure}[ht!]
  \centering   
 \begin{forest}
[$v_0$,
[$v_1$,
[][$v_2$,
[]
[$v_3$]
[]
][][]
]
[]
]
\end{forest}
  \caption{A fern of length three.}\label{fig:fern}
\end{figure}

We define $\mathrm{Fern}_{i,\alpha,j}(m)$ to be the set of quadruplets $\mathcal{T} = (V,E,\tau,\phi)$ such that
\begin{itemize}
\item The underlying graph $(V,E)$ is a fern of length $m$ with spine $\{v_0,\ldots,v_m\}$. 
\item The labelling function $\phi$ is a bijection between the leaves $L - \{v_m\}$ not equal to $v_m$ and the set $[\alpha]$.
\item The typing function $\tau:V \to [N]$ satisfies $\tau(v_0) = i$, $\tau(v_m) = j$, and for all $v \in L - \{v_m\}$, if $\phi(v) = (k,a)$ for some $a$ then $\tau(v) = k$. 
\end{itemize}

The following lemma states the nilpotency of $H$ in terms of energy sums over sets of ferns.

\begin{lemma}[The fern lemma] \label{lem:fern}
Let $H:\mathbb{C}^N \to \mathbb{C}^N$ be a polynomial mapping containing only terms of degree two and higher. The following are equivalent.
\begin{enumerate}
\item The $m^{\text{th}}$ power of the Jacobian matrix of $H$, $J(H)^m$, is equal to the zero matrix on $\mathbb{C}^N$. 
\item For every $i,j,\alpha$,
\begin{align*}
\sum_{ \mathcal{T} \in \mathrm{Fern}_{i,\alpha,j}(m) } \mathcal{E}_H( \mathcal{T} ) = 0.
\end{align*}
\end{enumerate}
\end{lemma}

We remark that an idea similar to Lemma \ref{lem:fern} appears in the work of Singer \cite[Section 3]{singer1}, who studies the special case where $H$ is quadratic. Lemma \ref{lem:fern} is proved in Section \ref{sec:JC proofs}.

The main result of this section is the following reformulation of the Jacobian conjecture in purely combinatorial terms. This reformulation is based on combining Lemma \ref{lem:fern}, Theorem \ref{thm:BCW}, and the definition of the involution $\Phi$ described below Theorem \ref{thm:inversion}.

\begin{conj}[A combinatorial form of the Jacobian conjecture] \label{JC3}
There exists an integer $\delta \geq 3$ with the following property. Suppose $N \geq 1$ and \[H := \left( H_{i,\alpha} : i \in [N] , \alpha \in \mathbb{Z}_{ \geq 0}^N, |\alpha| = \delta \right)\] is a collection of complex numbers with the property that there exists $m \leq N$ such that for every $i,j \in [N]$ and every $\alpha \in \mathbb{Z}_{\geq 0}^N$, 
\begin{align*}
\sum_{ \mathcal{T} \in \mathrm{Fern}_{i,\alpha,j}(m) } \mathcal{E}_H( \mathcal{T} ) = 0.
\end{align*}
Now define a second collection of complex numbers
\begin{align*}
\Phi(H)_{i ,\alpha} := \sum_{ \mathcal{T} \in \mathbb{S}_{i,\alpha} } \mathcal{E}_H(\mathcal{T} ).
\end{align*}
Then only finitely many of the $\Phi( H)_{i,\alpha}$ are non-zero.
\end{conj}

To clarify, if Conjecture \ref{JC3} is true for any $\delta \geq 3$, then the Jacobian conjecture holds. 

We now make a few further remarks about further assumptions we can make in the setting of Conjecture \ref{JC3} based on the various reductions of the Jacobian conjecture:
\begin{itemize}
\item According to Dru\.zkowski's strengthening \cite{druzkowski} of the Bass, Connell and Wright reduction, we may assume further that $H$ is the cube of a linear map. This amounts to the existence of a matrix $\left( L_{i,j} \right)_{1 \leq i,j \leq N}$ such that each $H_{i,\alpha}$ has the form
\begin{align*}
H_{i,\mathbf{e}_a + \mathbf{e}_b + \mathbf{e}_c } = L_{i,a} L_{i,b} L_{i,c}.
\end{align*}
In fact, in the later work \cite{druzkowski2}, Dru\.zkowski showed that we may further assume $L^2 = 0$.
\item If we appeal to the alternative strengthening of the Bass, Connell and Wright result by de Bondt and Van den Essen \cite{DBVDE}, we may assume that $H$ has a symmetric Jacobian. By the Poincar\'e lemma, this amounts to the existence of a second collection of constants $(J_\beta : \beta \in \mathbb{Z}_{ \geq 0}^N : |\beta| = 4)$ such that
\begin{align*}
H_{i, \alpha } = J_{ \alpha + \mathbf{e}_i}.
\end{align*}
See for instance Zhao \cite{zhao}. Wright has studied analogous statements to Conjecture \ref{JC3} in the symmetric case --- though utilising a different power series inversion formula --- finding that the symmetry of the Jacobian matrix allows one to reformulate the problem in terms of unrooted trees \cite{wright3,wright4}.
\end{itemize}

It is worth emphasising as mentioned above that the Jacobian conjecture is true for quadratic maps \cite{lcwang}; equivalently, the $\delta = 2$ version of Conjecture \ref{JC3} is true. Unfortunately, to date there are no known combinatorial proofs of this fact, though Singer \cite{singer1,singer2,singer3} has made some inroads. We hope the formulation we have presented here will yield deeper insight into possible combinatorial approaches to the problem.

\section{The non-commutative case} \label{sec:free}

In this section, we now present the main results of this article concerning power series in $N$ free variables $X_1,\ldots,X_N$ such that in general
\begin{align*}
X_i X_j \neq X_j X_i.
\end{align*}

\subsection{Rings of power series in free variables and a free Fa\`a di Bruno formula}

Suppose again that $\mathbb{K}$ is a commutative ring, and let $\overrightarrow{R} := \mathbb{K} \langle \langle X_1,\ldots, X_N \rangle \rangle$ be the ring of formal power series in $N$ free indeterminates with coefficients in $\mathbb{K}$. Each element $f$ of $\overrightarrow{R}$ has the form
\begin{align*}
f ( X_1,\ldots, X_N) = \sum_{ k \geq 0} \sum_{  \kappa \in [N]^k } f_{  \kappa}  X_ \kappa
\end{align*}
where for $ \kappa = ( \kappa_1,\ldots, \kappa_k) \in [N]^k$, $X_ \kappa := X_{ \kappa_1} \ldots X_{ \kappa_k }$, and we use the convention that $[N]^0 := \{ \varnothing \}$, where $\varnothing$ is the empty tuple and $X_{\varnothing}$ refers to the monic constant polynomial in $\fR$. 

We have the following free analogue of the Leibniz rule, namely that if 
\begin{align*}
f(X_1,\ldots,X_N) g( X_1,\ldots, X_N) := \sum_{ k \geq 0} \sum_{ i \in [N]^k } h_ \kappa X_ \kappa,
\end{align*}
then each $h_ \kappa$ is given by 
\begin{align*}
h_{ \kappa_1,\ldots, \kappa_k } = \sum_{0 \leq j \leq k} f_{ \kappa_1,\ldots, \kappa_j} g_{ \kappa_{j+1}, \ldots,  \kappa_k }.
\end{align*}

\begin{rem} \label{rem:free}
In studying the composition of power series in free variables, the object we had in mind were polynomials in several matrix variables with complex coefficients. Indeed, we would like to take a moment to emphasise a few structural apsects of our set up:
\begin{itemize}
\item The coefficients $f_\kappa$ of our power series take values in a commutative ring.
\item The variables $X_1,\ldots,X_N$ are not assumed to commute with one another.
\item The coefficients commute with the variables.
\end{itemize}
The first point above is where our set up differs from related work by Brouder, Frabetti and Krattenthaler \cite{BFK} and Anshelevich, Effros and Popa \cite{AEP}. In both of these works on non-commutative power series, the coefficients themselves take values in a non-commutative algebra, which gives rise to significantly differences which we discuss at the end of this section.
\end{rem}

In order to give our free Fa\`a di Bruno formula, we require definitions surrounding \emph{planar trees}. A rooted planar tree is a triplet $(V,E,r)$ where $(V,E)$ is a rooted tree, and $r$ is a function ranking the children of each vertex: so that whenever $v$ is a vertex and the set $\{ w : \text{$w$ is a child of $v$} \}$ of children of $v$ has cardinality $k$, $r$ is a bijection between the set of children of $v$ and $\{1,\ldots,k\}$. The function $r$ gives rise to a planar embedding of the tree: by drawing the children of a vertices from left to right according to their rankings. Two planar trees $(V,E,r)$ and $(V',E',r')$ are considered isomorphic if there is a rank-preserving graph isomorphism between the underlying vertex sets.

Suppose a rooted planar tree $(V,E,r)$ has $k$ leaves. The ranking induces an ordering on the leaves of the vertex set: so that we label the leaves $\{v_1,\ldots,v_k \}$ according to their positions clockwise in the plane; see Figure \ref{fig:planar tree}. We refer to $v_j$ as the $j^{\text{th}}$ leaf.

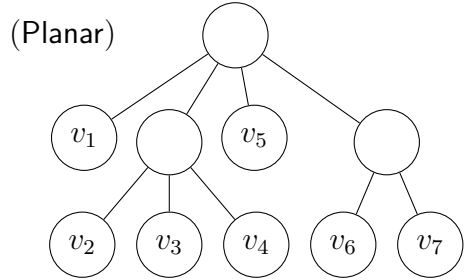
\begin{figure}[ht]
  \centering   
 \begin{forest}
[,name=L1,
[$v_1$, name=A]
[
[$v_2$]
[$v_3$]
[$v_4$]
]
[$v_5$]
[
[$v_6$]
[$v_7$, name=B]
]
]
\node (a) [left=of L1] {($\mathsf{Planar}$)};
\end{forest}
  \caption{A planar tree with seven leaves.}\label{fig:planar tree}
\end{figure}

We will be interested in labelled rooted planar trees. Let $\kappa = (\kappa_1,\ldots,\kappa_k) \in [N]^k$ be a $k$-tuple of elements of $[N]$, and let $\fT_{i,\kappa}$ denote the set of labelled rooted planar trees with root type $i$ and leaf types $\kappa$. In other words, $\fT_{i,\kappa}$ consists of quadruplets $\freeT = (V,E,r,\tau)$ such that
\begin{itemize}
\item The triplet $(V,E,r)$ is a rooted planar tree.
\item The type function $\tau$ is any function $\tau:V \to [N]$ such that if $v_j$ is the $j^{\text{th}}$ leaf of $(V,E,r)$, then $\tau(v_k) = \kappa_j$. 
\end{itemize} 
In analogy to the commutative case, we will consider three subsets of $\fT_{i,\kappa}$:
\begin{itemize}
\item The subset $\fS_{i,\kappa}$ of $\fT_{i,\kappa}$ consisting of the proper trees --- i.e. those trees where every internal vertex has two or more children.
\item The subset $\fF_{i,\kappa}(m)$ of $\fT_{i,\kappa}$ consisting of the final trees of length $m$ --- i.e. those labelled planar trees where every leaf occurs in generation $m$. 
\item The subset $\fA_{i,\kappa}$ of $\fT_{i,\kappa}$ consisting of the alternating trees --- i.e. those trees where every vertex in an even generation has exactly one child, and everyone vertex in an odd generation is either a leaf or has two or more children. 
\end{itemize}

Our free Fa\`a di Bruno's formula in $N$ non-commutative variables is given in terms of final labelled planar trees, and our inversion formula for power series in $N$ non-commutative variables is given in terms of proper labelled planar trees. See for example Figure \ref{fig:quick free tree} above, which depicts an element of $\fT_{5,\kappa}$, where $\kappa = (1,2,1,3)$. 

We now define energies for non-commutative trees. Unsurprisingly perhaps, the energies themselves also depend on the planar structure of the tree. For an internal vertex $v$ of a labelled rooted planar tree with its children listed in ranked order $w_1,\ldots,w_k$, we define the \emph{free outdegree} of $v$ to be the tuple $\fmu(v) := (\tau(w_1),\ldots,\tau(w_k)$. See Figure \ref{fig:v2}.

\begin{figure}[ht!]
  \centering   
 \begin{forest}
[3, fill=col3, name=L1,
[3, fill=col3]
[4, fill=col4]
[1, fill=col1]
]
\node (a) [left=of L1] {($\mathsf{Planar}$)};
\end{forest}
  \caption{An internal vertex $v$ in a labelled planar tree. This vertex $v$ has type $\tau(v) = 3$ and free outdegree $\protect\fmu(v) = (3,4,1)$. }\label{fig:v2}
\end{figure}
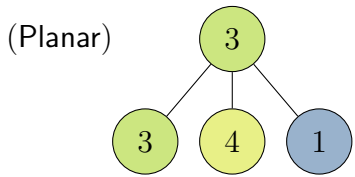

Let $\fR^N$ denote the set of $F = (F_1,\ldots,F_N)$, where each component $F_i$ of $F$ is an element of $\fR$. Suppose $F^* = ( F^{(1)},\ldots,F^{(m)})$ is a sequence of elements of $\fR^N$. We define the non-commutative $F^*$-energy of a labelled planar tree $\freeT = (V,E,r,\tau)$ in $\fF_{i,\kappa}(m)$ by 
\begin{align*}
\mathcal{E}_{F^*} \left( \freeT \right) = \prod_{ l = 1}^m \prod_{ v \in V_{l-1} } F^{(l)}_{ \tau(v), \fmu(v) } .
\end{align*}

We write $\fR^N_0$ for the subset of $\fR^N$ consisting of power series with no constant term, i.e. $F$ such that $F_i(0,0,\ldots,0) = 0$ for each $i \in [N]$. We are now equipped to state our free version of Fa\`a di Bruno's formula. 

\begin{thm} \label{thm:free fdb}
Suppose that $F^{(1)},\ldots,F^{(m)}$ are elements of $\fR^N_0$, so that the $i^{\text{th}}$ component of $F^{(l)}$ is given by 
\begin{align*}
F^{(l)}_i(X_1,\ldots,X_N) := \sum_{ k \geq 1} \sum_{ \kappa \in [N]^k } F_{i,\kappa} X_\kappa.
\end{align*}
Then the coefficient of $X_\kappa$ in the $i^{\text{th}}$ component of the composition $F^{(1)} \circ \ldots \circ F^{(m)}$ is given by 
\begin{align*}
\left( F^{(1)} \circ \ldots \circ F^{(m)} \right)_{ i , \kappa } =  \sum_{ \mathcal{T} \in \fF_{i,\kappa}( m) } \mathcal{E}_{F^*} \left( \freeT \right) .
\end{align*}
\end{thm}

\begin{example}

Suppose $F$ and $G$ are elements of $\fR^N$, and we would like to find the coefficient $(F \circ G)_{2, (3,1,1)}$ of $X_3 X_1^2$ in the $2^{\text{nd}}$ component of the composition $F \circ G$. Then according to Theorem \ref{thm:free fdb}, $( F \circ G)_{2, (3,1,1) }$ may be written as a sum over trees in $\fF_{2,(3,1,1)}(2)$. Figure \ref{fig:v3} depicts the set of trees in $\fF_{2,(3,1,1)}$ as the types $a,b,c$ of the internal vertices range over $[N]$.

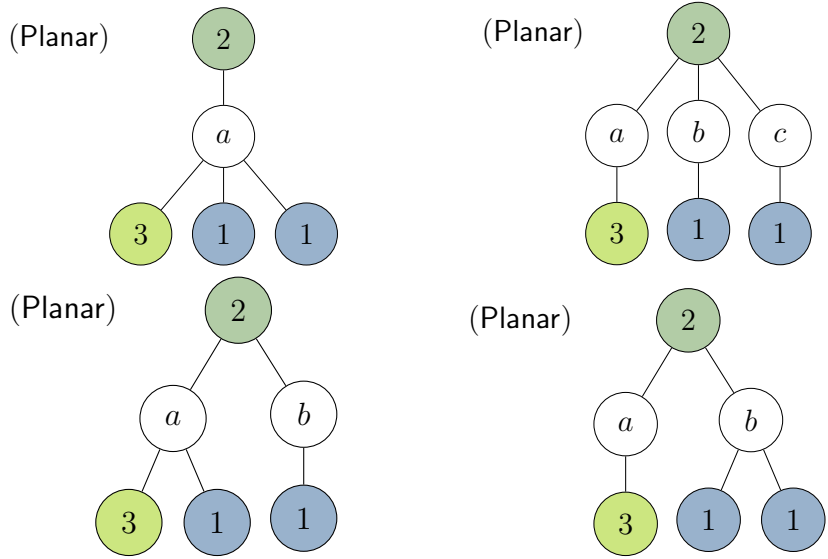
\begin{figure}[htbp]
\begin{tabular}{C{.32\textwidth}C{.32\textwidth}}
{
\resizebox{0.25\textwidth}{!}{%
\begin{forest}
[2, fill=col2, name=L1,
[$a$,
[3, fill=col3]
[1, fill=col1]
[1, fill=col1]
]
]
\node (a) [left=of L1] {($\mathsf{Planar}$)};
\end{forest}
}}&
{
    \resizebox{0.25\textwidth}{!}{%
\begin{forest}
[2, fill=col2, name=L2,
[$a$,
[3, fill=col3]
]
[$b$,
[1, fill=col1]
]
[$c$,
[1, fill=col1]
]
]
\node (a) [left=of L2] {($\mathsf{Planar}$)};
\end{forest}
}}\\ 
{
    \resizebox{0.25\textwidth}{!}{%
\begin{forest}
[2, fill=col2,name=L3,
[$a$,
[3, fill=col3]
[1, fill=col1]
]
[$b$,
[1, fill=col1]
]
]
\node (a) [left=of L3] {($\mathsf{Planar}$)};
\end{forest}
}}& 
{
    \resizebox{0.27\textwidth}{!}{%
\begin{forest}
[2, fill=col2, name=L4,
[$a$,
[3, fill=col3]
]
[$b$,
[1, fill=col1]
[1, fill=col1]
]
]
\node (a) [left=of L4] {($\mathsf{Planar}$)};
\end{forest}
}}
\end{tabular}
\caption{The four possible planar topologies for trees in $\protect\fF_{i,\kappa}(2)$, where $i = 2$ and $\kappa = (3,1,1)$. The types $a,b,c$ may take any value in $\protect[N]$. } \label{fig:v3}
\end{figure}
In particular, if $F,G$ are elements of $\fR^N$, then 
\begin{align*}
(F \circ G)_{ 2, (3,1,1) } = &\sum_{ a \in [N] } F_{2,a} G_{a, (3,1,1)} + \sum_{ a,b,c \in [N] } F_{2,(a,b,c)} G_{a,3} G_{b,1} G_{c,1}\\
& + \sum_{ a,b \in [N] } G_{ 2, (a,b) } F_{a, (3,1)} F_{b,1} + \sum_{ a,b \in [N] } F_{2,(a,b) } G_{a,3} G_{b, (1,1)}.
\end{align*}

\end{example}

\subsection{Inversion of power series in free variables}

We now turn to compositional inversion of elements of $\fR^N$. Consider first of all the compositional identity element of $\fR^N$ --- the element $I$ of $\fR^N$ whose $i^{\text{th}}$ component is given by $X_i$. Namely given an element $F$ of $\fR^N$, we are interested in identifying whether $F$ has a compositional inverse $G$ satisfying $F \circ G = G \circ F = X$, and in identifying the coefficients of the inverse. 

Like in the commutative case, we begin by considering the case where the linear term 
\begin{align*}
J(F)(0) := \left( F_{i,(j)} \right)_{1 \leq i,j \leq N}
\end{align*}
of $F$ is the identity matrix. Indeed, given a free $N$-dimensional power series whose $i^{\text{th}}$ component has the form $H(X_1,\ldots,X_N) := \sum_{ k \geq 2} \sum_{ \kappa \in [N]^k } H_{i, \kappa} X_\kappa$, define the planar $H$-energy of a planar tree $\freeT$ by 
\begin{align*}
\mathcal{E}_H \left( \freeT \right) := \prod_{ v \in I} H_{\tau(v),\fmu(v)}.
\end{align*}

\begin{thm} \label{thm:free inversion}
Let $F$ be an element of $\fR^N$ whose $i^{\text{th}}$ component has the form
\begin{align*}
F_i(X_1,\ldots,X_N) = X_i - \sum_{ k \geq 2} \sum_{ \kappa \in [N]^2 } H_{i,\kappa} X_\kappa.
\end{align*}
Then there exists an element $G$ of $\fR^N$ such that $G$ is a compositional inverse of $F$ in that $F \circ G = G \circ F = X$. Moreover, $G$ has the form
\begin{align*}
G_i(X_1,\ldots,X_N) = X_i + \sum_{ k \geq 2} \sum_{ \kappa \in [N]^k } G_{i,\kappa} X_\kappa,
\end{align*}
where the coefficient $G_{i,\kappa}$ is given in terms of the sum
\begin{align*}
G_{i,\kappa} = \sum_{ \freeT \in \fS_{i,\kappa} }  \mathcal{E}_H \left( \freeT \right).
\end{align*}
\end{thm}

Where Theorem \ref{thm:free inversion} was concerned with inverting power series with identical linear terms, we now turn to the case where the linear term is non-identical, which we find parallels the commutative case. Given a matrix $(Q_{i,\mathbf{e}_j})$, and a collection of elements $\left( H_{i,\kappa} : i \in [N], \kappa \in [N]^k, k \geq 2 \right)$, we define the planar $(H,Q)$ energy of an alternating planar tree $\freeT$ by 
\begin{align*}
\mathcal{E}_{H,Q}( \freeT ) := \prod_{ \text{$v$ even} } Q_{\tau(v),\fmu(v)} \prod_{ \text{$v$ odd}} H_{\tau(v), \fmu(v) }.
\end{align*}

\begin{thm} \label{thm:free inversion general}
Let $F$ be an element of $\fR^N$ whose $i^{\text{th}}$ component has the form
\begin{align*}
F_i(X_1,\ldots,X_N) = \sum_{ j= 1}^N P_{i,j} X_j - \sum_{ k \geq 2} \sum_{ \kappa \in [N]^k} H_{i,\kappa} X_\kappa.
\end{align*}
Then if $P$ is an invertible matrix, then $G$ has a composition inverse. Moreover, if $Q$ is the matrix inverse of $P$, then the $i^{\text{th}}$ component of $G$ has the form
\begin{align*}
G_i(X_1,\ldots,X_N) = \sum_{ j  = 1}^N Q_{i,j} X_j + \sum_{ k \geq 2} \sum_{ \kappa \in [N]^k } G_{i,\kappa} X_\kappa
\end{align*}
where for $k \geq 2$ and $\kappa \in [N]^k$, the coefficients of $G$ are given by 
\begin{align*}
G_{i,\kappa} := \sum_{ \freeT \in \fA_{i,\kappa} } \mathcal{E}^{\mathsf{alt}}_{H,Q} ( \freeT ).
\end{align*}
\end{thm}

We conclude this section by discussing further the aforementioned work of Brouder, Frabetti and Krattenthaler \cite{BFK} and Anshelevich, Effros and Popa \cite{AEP}, who study the composition and inversion of non-commutative power series in non-commuting variables through the apparatus of incidence algebras. The former paper \cite{BFK} is concerned with a single-variable and the latter \cite{AEP} with many variables. As we said above, the set ups of these works differ from ours in that they assume the coefficients themselves of the series take values in a non-commutative algebra --- though the coefficients are assumed to commute with the variables. The non-commutativity of the coefficients gives rise to several interesting structural differences. Most notably, the `composition' of power series need not be associative, so that when the coefficients are non-commutative, we have
\begin{align*}
\left( F \circ G \right) \circ H \neq F \circ \left( G \circ H \right)
\end{align*}
in general. (To use the example supplied in \cite{AEP}, suppose $N = 1$ and consider the power series $F(X_1) = X_1^2, G(X_1) = b X_1$ and $H(X_1) = aX_1$, where the coefficients $a$ and $b$ do not commute.) An interesting further consequence of having non-commutative coefficients is that right and left compositional inverses of a power series need not coincide \cite[Corollary 11]{AEP}. In fact, Anshelevich et al. \cite{AEP} are hesitant to use the word composition for these reasons, preferring the term \emph{power series substitution}. In summary, Theorem \ref{thm:free inversion} is connected with \cite[Theorem 8]{AEP}, with various reductions and symmetries arising in our case thanks to the commutativity of the power series coefficients.

More broadly, the study of polynomials and power series in non-commutive variables and with non-commuting coefficients dates back at least to the start of the $20^{\text{th}}$ century (see e.g. Ore \cite{ore}). 
To this day this remains an active research area, with Fa\`a di Bruno's formula appearing in a multitude of non-commutative settings. This in constrast to our set up ---- see Remark \ref{rem:free}. On this front we also mention the recent work of Frabetti and Shestakov \cite{FS}, who among other things obtain a generalisation of the Lagrange-inversion formula for series with non-commutative coefficients.

\subsection{A non-commutative Jacobian conjecture}
We conclude the discussion of power series in free variables by exploring possible non-commutative analogues to the Jacobian conjecture. 

Recall that in the commutative setting, the Jacobian conjecture asserts that in order for polynomial mapping $F \in R^N$ with coefficients in $\mathbb{C}$ to have a polynomial compositional inverse, it is sufficient to check that the determinant of the Jacobian matrix is equal to a non-zero complex constant. When we attempt to develop a parallel statement in the case where the variables $X_1,\ldots,X_N$ do not commute, we find that certain structural aspects of the problem are different.

Indeed, like in the commutative case we say an element $F = \left(F_{i,\kappa} : i \in [N], \kappa \in [N]^k , k \geq 0 \right)$ of $\fR^N$ is a polynomial mapping if only finitely many of the coefficients $(F_{i,\kappa})$ are non-zero.
With a view to defining the Jacobian matrix associated with a polynomial mapping in free variables, we first need a notion of differentiation. To this end, given a multi-index $\kappa= (\kappa_1,\ldots,\kappa_k)$, define
\begin{align*}
\frac{ \partial}{ \partial X_j } X_{\kappa_1} \ldots X_{\kappa_k} := \sum_{ l = 1}^k \ind_{ \kappa_l = j} X_{i_1} \ldots X_{i_{l-1}} X_{i_{l+1}} \ldots X_{i_k}.
\end{align*}
For example,
\begin{align*}
\frac{ \partial}{ \partial X_2 } X_2 X_2  X_1 X_2 X_4   = 2 X_2 X_1 X_2 X_4  + X_2 X_2 X_1 X_4 .
\end{align*}
A similar definition for differentiating non-commutative monomials appears in Rota, Sagan and Stein \cite{RSS}, where they refer the operation as the \emph{Hausdorff derivative}. By linearity the operator $\frac{ \partial}{ \partial X_j}$ extends to the ring of formal power series in $N$ free indeterminates. 

Again, for the sake of concreteness, let the coefficient ring $\mathbb{K} = \mathbb{C}$, and suppose we have a power series $F \in \fR^N$ in $N$ free variables with coefficients in $\mathbb{C}$. We can then define its Jacobian matrix to be the $N \times N$ matrix $\left( \frac{ \partial F_i}{ \partial X_j }  \right)_{ 1 \leq i,j \leq N }$ with coefficients in the (non-commutative) ring $\fR$. Here is the point at which the non-commutative case diverges with the commutative case: the invertibility of this Jacobian matrix at a point $(X_1,\ldots,X_N)$ may not be expressed straightforwardly in terms of a determinant because the ``determinant'' of a matrix with non-commuting entries is ill-defined.

By Theorem \ref{thm:free inversion general}, we know that a polynomial mapping $F$ in $\fR^N$ is guaranteed to have a compositional inverse $G$ in $\fR^N$ if and only if its Jacobian matrix is invertible at the origin. However, it is is by no means clear to the authors under what conditions the inverse $G$ is itself a polynomial. This discussion leads us pose the following question.

\begin{que}
Let $F$ be a polynomial mapping in $\fR^N$ in $N$ non-commuting variables whose Jacobian matrix is invertible at the origin so that $F$ has a power series inverse $G$. When is the inverse free power series $G$ itself a polynomial?
\end{que}

One possible way forward in terms of formulating a statement analogous to the Jacobian conjecture in the free case would be in terms of \emph{quasi-determinants}, which serve as analogues to the classical determinant for matrices with non-commutative entries. In their present form, quasi-determinants were introduced by Gel'fand and Retakh in \cite{GR}. We refer the reader to Gel'fand, Gel'fand, Retakh and Wilson \cite{GGRW} for a survey. 

\section{A few applications} \label{sec:applications}

Before proving our main results in Sections \ref{sec:fdbproof} and \ref{sec:lagrangeproof}, in this section we discuss a few applications of Fa\`a di Bruno's formula and our inversion formula to various areas of mathematics, including the Stirling numbers, the Hermite polynomials, the cumulants of random variables, reciprocation of power series, enumeration of trees, and the genealogy of Galton-Watson trees. 

\subsection{Stirling numbers}
Let $B(k,j)$ denote the number of of set partitions the set $\{1,\ldots,k\}$ into $j$ blocks. (So that in particular, $B(k,j)$ is zero when $j > k$.) The numbers $B(k,j)$ are known as the Stirling numbers of the second kind. We may use the one-dimensional Fa\`a di Bruno to obtain a quick derivation of the joint generating function of the sequence $\left( B(k,j) \right)_{k, j \geq 1 }$. Indeed, consider the function $h(x) := \exp \left( a e^{x} - a \right)$, which we may regard as the composition $h = f \circ g$ of the functions
\begin{align*}
f(x) = e^{ax} \qquad \text{and} \qquad g(x) : = e^{ x } - 1.
\end{align*}
By using the Fa\`a di Bruno formula \label{eq:faa di bruno}, we immediately see that 
\begin{align} \label{bell}
\frac{ d^k }{ dx^k } \exp \left( a e^{x} - 1 \right) \Big|_{x= 0} &= \sum_{ \pi \in \mathcal{P}_k } a^{ \# \pi}.
\end{align}
In particular, by summing \eqref{bell} over $k$, we have
\begin{align*}
\exp \left( a e^{bx} - 1 \right) = \sum_{ k = 0}^\infty \frac{b^k}{k!}  \sum_{ \pi \in \mathcal{P}_k } a^{ \# \pi } = \sum_{ k,j \geq 0} B(k,j) a^j \frac{b^k}{k!},
\end{align*} 
which gives a joint generating function of the Stirling numbers of the second kind $\left( B(k,j) \right)_{k,j \geq 1}$.

\subsection{Hermite polynomials}
The Hermite polynomials $\left( H_n(x) \right)_{ n \geq 0 }$ are a sequence of classical orthogonal polynomials arising in numerous applications across mathematics. The Hermite polynomials satisfy the orthogonality property
\begin{align*}
\int_{ - \infty}^\infty H_i(s) H_j(s) \frac{e^{ - s^2/2} ds }{  \sqrt{2 \pi} } = \delta_{i,j} i!,
\end{align*}
and may be defined through the higher derivatives
\begin{align*}
H_k(x) := (-1)^n e^{ x^2/2 } \frac{ d^k }{ dx^k } e^{ - x^2/2}.
\end{align*}
We now use the one dimensional Fa\`a di Bruno formula to obtain a new representation for the Hermite polynomials in terms of partitions of $\{1,\ldots,k\}$ into sets of size at most $2$. Indeed, setting $f(x) = e^{x}$ and $g(x) = - x^2/2$ in the Fa\`a di Bruno formula we obtain
\begin{align} \label{eq:hermite}
H_k(x) = \sum_{ \pi \in \mathcal{P}_k } (-1)^{n + \# \pi} \prod_{ \Gamma \in \pi } \left( \ind_{ \# \Gamma = 1} x + \ind_{ \# \Gamma = 2} \right).
\end{align}
The formula \eqref{eq:hermite} states that the Hermite polynomials have a representation as a \emph{matching polynomial}; see for instance Godsil \cite{godsil}.

\subsection{Cumulants}
Let $Y = \left( Y_1,\ldots, Y_N \right)$ be an $\mathbb{R}^N$-valued random variable, and consider the moment and cumulant generating functions 
\begin{align*}
M(s_1,\ldots,s_N) := \mathbb{E} [ e^{s_1 Y_1 + \ldots + s_N Y_N } ] \qquad \text{and} \qquad K(s_1,\ldots,s_N) := \log M(s_1,\ldots,s_N).
\end{align*}
For simplicity, suppose that $Y$ has moments of all orders, so that $M$ and $K$ may be written as a power series
\begin{align*}
M(s) = 1 + \sum_{ \alpha \neq 0} \frac{ m_\alpha}{ \alpha!} s^\alpha \qquad  \text{and} \qquad K(s) = \sum_{ \alpha \neq 0 } \frac{ \kappa_\alpha}{\alpha!} s^\alpha.
\end{align*}
Of course, each coefficient $m_\alpha$ is equal to the moment $\mathbb{E} [ Y^\alpha]$ of $Y$. The coefficients $(\kappa_\alpha)$ of $K$ on the other hand are known as the \emph{cumulants} of $X$. Though the moments are more natural than the cumulants, there are many occasions in probability where it is more convenient to work with the cumulants. Take for instance the fact that the standard \text{one-dimensional} Gaussian distribution is characterised by the succinct property that it is the unique random variable with cumulants $\kappa_n := \ind_{ n =2}$. We now address the problem of computing the cumulants in terms of the moments, and vice versa.

By the special case $m=2$ of Theorem \ref{thm:main}, and the following Remark \ref{rem:main}, if $f:\mathbb{R} \to \mathbb{R}$ and $G: \mathbb{R}^N \to \mathbb{R}$, the derivatives of $(f \circ G):\mathbb{R}^N \to \mathbb{R}$ are given by 
\begin{align} \label{eq:special}
D^{\alpha}[ f \circ G ] (s) = \sum_{ \mathcal{T} \in \mathbb{F}_{1,\alpha}(2) }D^{ \mu(v_0) }[ f] \circ G(s)  \prod_{v \in V_1 } D^{ \mu(v) } [G](s).
\end{align}
Now using the simple facts  
\begin{align*}
\left( \frac{d}{dx} \right)^j \log(x) \Big|_{x = 1} = ( -1)^{j-1} (j-1)!  \qquad \text{and} \qquad \left( \frac{d}{dx}\right)^j \exp(x) \Big|_{x = 0}  = 1,
\end{align*}
in conjunction with respectively setting $(f,G) = (\log , M)$ and $(\exp, K )$ and $s = 0$ in \eqref{eq:special}, we obtain the following formulas allowing us to pass between the multivariate moments and cumulants:
\begin{align} \label{mcpass}
\kappa_\alpha :=  \sum_{ \pi \in \mathcal{P}_{[\alpha]} } (-1)^{ \# \pi - 1} \left( \# \pi - 1 \right)! \prod_{ \Gamma \in \pi } m_{ \# \Gamma}  \qquad \text{and} \qquad 
m_\alpha :=  \sum_{ \pi \in \mathcal{P}_{[\alpha]} } \prod_{ \Gamma \in \pi } \kappa_{ \# \Gamma} .
\end{align}
The derivation of \eqref{mcpass} using Fa\`a di Bruno's formula is a slick alternative to the direct combinatorial approach often used in the literature (see e.g. Exercise 4 of Section 1.1 in Mingo and Speicher \cite{MS}).

\subsection{Reciprocals of power series}

Suppose we have a formal power series 
\begin{align*}
f(X) = 1 + \sum_{ |\alpha| \neq 0} \frac{ f_\alpha}{ \alpha!} X^\alpha.
\end{align*}
It is straightforward to show using \eqref{eq:special} that $\frac{1}{f(X)} := 1 + \sum_{ \alpha \neq 0} \frac{ h_\alpha}{\alpha!} X^\alpha$ 
where for $\alpha \neq 0$,
\begin{align*}
h_\alpha = \sum_{ \pi \in \mathcal{P}_{[\alpha]} } (-1)^{ \# \pi } \# \pi! \prod_{ \Gamma \in \pi} f_{\# \Gamma}.
\end{align*}

\subsection{Enumeration of trees}

For positive integers $k \geq 1$, recall the set $\mathbb{S}_k$ of rooted proper trees with leaves in bijection with $[k]$ introduced in Section \ref{sec:introduction}. Theorem \ref{thm:inversion1} supplies us with a quick way to enumerate $\mathbb{S}_k$.

Indeed, in the context of Theorem \ref{thm:inversion1} consider the power series $h$ such that $\mathcal{E}_h(\mathcal{T})  =1$ for every tree $\mathcal{T}$. We see that the unique power series with no constant or linear term giving this property is $h(X) = e^X - 1 - X$. It is then an immediate consequence of Theorem \ref{thm:inversion1} that 
\begin{align*}
\# \mathbb{S}_k = \text{Coefficient of $\frac{X^k}{k!}$ in power series inverse of $2X + 1 - e^X$}.
\end{align*}
In other words, the exponential generating function of $\mathbb{S}_k$ is given by the functional inverse of $f(X) = 2X +1 - e^X$. This result appears in Stanley \cite{stanley2}. 

We may now treat the energy functional $\mathcal{E}_h(\mathcal{T})$ more generally as a partition function counting the number of vertices of certain degrees with trees of $\mathbb{S}_k$. For instance, suppose we want to count the size of the subset $\mathbb{S}_k^{\mathsf{even}}$ of $\mathbb{S}_k$ consisting only of trees in which vertices have an even number of children. Then we look at the power series 
\begin{align*}
h(X ) = \sum_{ k \geq 2} \frac{ \ind_{ \text{ $k$ even } } }{ k!} X^k  = \cosh(X) - 1,
\end{align*} 
As a result, we have 
\begin{align*}
\# \mathbb{S}_k^{\mathsf{even}} = \text{Coefficient of $\frac{X^k}{k!}$ in functional inverse of $2X + 1 - \cosh(X)$}.
\end{align*}

\subsection{The genealogical structure of Galton-Watson trees}

Fa\`a di Bruno's formula is required to understand the genealogical structure of Galton-Watson trees, which are stochastic processes modelling population growth which we now define. Before setting this up, we note that whenever $f,g$ are smooth functions with non-negative derivatives of all orders, and $D^n (f \circ g)(x) > 0$, the function $P:\mathcal{P}_n \to \mathbb{R}$ given by 
\begin{align} \label{fdbpi}
P(\pi) := \frac{  D^{\#\pi}[f] (g (x)  \prod_{ \Gamma \in \pi} D^{ \# \Gamma}[g](x)  }{ D^n[ f \circ g ](x) }
\end{align}
defines a probability measure on the set $\mathcal{P}_n$ of partitions of $\{1,\ldots,n\}$. It turns out that random partitions with probability laws of the form \eqref{fdbpi} occur naturally in the study of Galton-Watson trees.

Let $(p_i)_{i \geq 0}$ be a collection of non-negative reals satisfying $\sum_{ i \geq 0} p_i = 1$. The (continuous-time) Galton-Watson tree with offspring distribution $(p_i)_{i \geq 0}$ is the stochastic process defined as follows. We start with a single particle at time zero who lives for a random period of time of length $\tau$, where $\tau$ is standard exponential so that
\begin{align*}
\mathbb{P} \left( \tau > t \right) = e^{ - t}.
\end{align*}
Upon death this particle is replaced by a random number of children $L$, where
\begin{align*}
\mathbb{P} \left( L = i \right) = p_i.
\end{align*}
Each one of these children then independently repeats the behaviour of their parent: living for a standard exponential amount of time and then being replaced by a random number of children upon death with probabilities $(p_i)_{i \geq 0}$. We make the simple observation that for every pair of times $t < T$, that each individual living at some time $T$ is descended from a unique ancestor living in the time $t$ population. 

Galton-Watson trees are best studied through their generating functions. For $s \in [0,1]$ let $f(s) := \sum_{ i \geq 0} p_i s^i$, and suppose that $N_t$ is the number of particles in the process at time $t$. Then the process generating function $F_t(s) := \mathbb{E} \left[ s^{N_t} \right]$ satisfies the partial differential equation
\begin{align*}
\frac{ \partial F}{ \partial t }  = f ( F) - F, ~~~ F_0(s) = s.
\end{align*} 
(See e.g. Athreya and Ney \cite{AN}.) Moreover, $F_t$ is a semigroup, in that $F_{t_1} \circ F_{t_2} = F_{t_1 + t_2}$. 

Fix $T > 0$, and consider now conditioning on the event that $\{ N_T \geq k \}$ and picking $k$ distinct individuals $u_1,\ldots,u_k$ uniformly from the population at time $T$. We may create a random partition $\pi_t$ of $\{1,\ldots,k\}$ by declaring $i$ and $j$ to be in the same block of $\pi_t$ if individuals $u_i$ and $u_j$ are descended from the same ancestor in the time $t$-population. According to Theorem 3.1 of \cite{johnston}, the law of $\pi_t$ is given by the integral formula
\begin{align} \label{eq:johnston}
\mathbb{P} \left( \pi_t = \pi \right) = \int_0^1 \Lambda^k_T(s) \frac{ F_t^{ \# \pi} (F_{T-t}(s) ) \prod_{ \Gamma \in \pi} F_{T-t}^{ \# \Gamma}(s) }{ F_T^k(s) } ds, 
\end{align}
where $F_t^j(s) = \left( \frac{ \partial}{ \partial s} \right)^j F_t(s)$, and 
\begin{align*}
\Lambda^k_T(s)ds := \frac{ (1-s)^{k-1} F_T^k(s) }{(k-1)!\mathbb{P} \left( N_T \geq k \right) }  ds
\end{align*}
is a probability measure for $s \in [0,1]$. We remark that the fact that the right-hand side of \eqref{eq:johnston} constitutes a probability measure on $\mathcal{P}_n$ is a consequence of the fact that $\Lambda_T^k(s)$ is a probability measure on $[0,1]$, and that the internal quotients are themselves probability measures by virtue of setting $f = F_t, g = F_{T-t}$ in the representation \eqref{fdbpi} and using the semigroup property $F_t \circ F_{T-t} = F_T$. 

We refer the reader to other appearances of  Fa\`a di Bruno's formula on work in the genealogical structure of branching processes: in the setting of Galton-Watson trees by Vatutin and co-authors \cite{LV,VHJ}, as well as continuous-state branching processes \cite{JL}.

That completes the section on applications. In the next section we work towards proving our generalisation of Fa\`a di Bruno's formula in the commutative case, Theorem \ref{thm:main}.

\section{Proof of the multivariate Fa\`a di Bruno formula} \label{sec:fdbproof}

In this section we work towards proving our main result in the commutative case, Theorem \ref{thm:main}, which is a generalisation of Fa\`a di Bruno's formula for a chain composition of functions in $N$ commutative variables. In Section \ref{sec:rules} we begin by making some first observations. In Section \ref{sec:fdb2} we prove Theorem \ref{thm:main} in the special case $m=2$, and in the following Section \ref{sec:fdb3} the result is proved for general $m$.

First, a word on notation: we will use the convention that in any expression, square brackets proceed composition, which proceeds multiplication. For instance
\begin{align*}
D[f] \circ g ~ h \text{ refers to }  \left( \left( D[f]  \right) \circ g \right)~ h.
\end{align*}

Our proof is based on a doubly inductive argument, first we prove the special case $m=2$ via an induction argument on $\alpha$. Then we prove the general case by induction on $m$. In the next section we investigate how the differentiation and projection operators interact with multiplication and composition of functions.

\subsection{Differentiation and combining functions}\label{sec:rules}
Let $\mathcal{A}$ be the set of smooth scalar-valued functions $f: \mathbb{R}^{\mathbb{N}} \to \mathbb{R}$, and $\mathcal{A}^N$ be the set of smooth vector-valued functions $F:\mathbb{R}^N \to \mathbb{R}^N$. Finally, let $\pi_j : \mathcal{A}^N \to \mathcal{A}$ be the projection of a function $F = (F_1,\ldots,F_N)$ onto its $j^{\text{th}}$ component $F_j$. 

We have several ways of combining functions in $\mathcal{A}$ or $\mathcal{A}^N$ to form a new function in one of these sets. First of all, we have addition. Namely given any pair of functions $f$ and $g$ in $\mathcal{A}$ (resp. $\mathcal{A}^N$), we may define a new function $f + g$ in $\mathcal{A}$ (resp. $\mathcal{A}^N$) by setting $(f + g) (s) := f(s) + g(s)$. Secondly, we have multiplication. Given any pair of functions $f$ and $g$ both in $\mathcal{A}$, we may define a function $fg$ in $\mathcal{A}$ defined by setting $f g(s) := f(s) g(s)$. Finally, we have composition. Namely, whenever $f$ is an element of $\mathcal{A}$ (resp. $\mathcal{A}^N$) and $g$ is an element of $\mathcal{A}^N$, we may define a function $f \circ g$ in $\mathcal{A}$ (resp. $\mathcal{A}^N)$ by setting $f \circ g(x) := f( g( x))$.

We now have a look at how first order differentiation $D^{\mathbf{e}_i} := \frac{ \partial}{ \partial x_i}$ interacts with each of these ways of combining functions. First of all, we have linearity. Namely, for $f_1,\ldots,f_p$ in $\mathcal{A}$ we have 
\begin{align} \label{linearity2}
D^{\mathbf{e}_i} \left[ \sum_{r=1}^p f_r \right] = \sum_{r=1}^p  D^{\mathbf{e}_i}[f_r].
\end{align}
As for multiplication, we have the product rule. For any $f,g$ in $\mathcal{A}$, we have $D^{\mathbf{e}^i} [fg] = D^{\mathbf{e}^i} [f] g + f D^{\mathbf{e}^i}  [ g]$ and more generally, it follows by induction that for $f_1,\ldots,f_p$ in $\mathcal{A}$,
\begin{align} \label{product2}
D^{\mathbf{e}_i}  \left[ \prod_{r = 1}^p f_r \right] = \sum_{ r = 1}^p D^{\mathbf{e}_i} [ f_r  ] \prod_{ s \neq r } f_s.
\end{align}
Finally, differentiation interacts with composition according to the chain rule. Namely whenever $f \in \mathcal{A}$ and $G$ is an element of $\mathcal{A}^N$, the first order derivatives of their composition satisfy
\begin{align} \label{chainrule2}
D^{\mathbf{e}_i} [ f \circ G ] = \sum_{ j = 1}^N \left( D^{\mathbf{e}_j } [f ] \circ G  \right) D_j^{\mathbf{e}_i}[G].
\end{align}
We will see that these three rules \eqref{linearity2}, \eqref{product2} and \eqref{chainrule2} are the building blocks in a proof of Theorem \ref{thm:main}. In the next section we start by proving the case $m = 2$ of Theorem \ref{thm:main}.

\subsection{The $m=2$ case: derivatives of $F \circ G$} \label{sec:fdb2}

On our way to proving Theorem \ref{thm:main}, first we prove the special case $m=2$, which we state as a lemma.

Recall that $\mathbb{F}_{i,\alpha}(2)$ is the set of labelled trees whose root $v_0$ has type $i$, the leaves are in bijection with $[\alpha]$, and every leaf of the tree lies in the second generation. 
\begin{lemma} \label{thm:main2}
For smooth functions $F,G:\mathbb{R}^N \to \mathbb{R}^N$, we have
\begin{align} \label{eq:scriabin}
D^{\alpha}_i [F \circ G ] = \sum_{ \mathcal{T} \in \mathbb{F}_{i,\alpha} ( 2) } \mathcal{E}_{F,G}(\mathcal{T} ) ,
\end{align}
where
\begin{align*}
\mathcal{E}_{F,G}(\mathcal{T} ) :=\left(  D_i^{\mu(v_0) } [F] \circ G \right) \prod_{ v \in V_1} D_{\tau(v)}^{\mu(v)} [G].
\end{align*}
\end{lemma}

\begin{proof}
We proceed by induction on the multi-index $\alpha$. 

The special case where $\alpha = \mathbf{e}_j$ follows immediately from the chain rule. Indeed, a tree with two generations and one leaf is simply a line with three vertices. The root has type $i$ and the sole leaf has type $j$, leaving $N$ choices of labelling for the internal vertex.

We now fix $\beta \in \mathbb{Z}_{ \geq 0}^N$, and show that if \eqref{eq:scriabin} holds for $\alpha = \beta$, then it holds for $\alpha = \beta + \mathbf{e}_j$. Indeed, using the inductive hypothesis in the second equality below, and the linearity of differentiation in the third, we have 
\begin{align} \label{symphony1}
D^{ \beta + \mathbf{e}_j }_i [ F \circ G ] &= D^{ \mathbf{e_j} } \left[ D^{ \beta}_i [F \circ G ] \right] \nonumber \\
&= D^{ \mathbf{e_j} } \left[     \sum_{ \mathcal{T} \in \mathbb{F}_{i,\alpha} ( 2) } \mathcal{E}_{F,G}(\mathcal{T} )  \right] \nonumber \\
&=  \sum_{ \mathcal{T} \in \mathbb{F}^{[\beta]}(2) }   D^{ \mathbf{e_j} } \left[         \mathcal{E}_{F,G}(\mathcal{T} )  \right].
\end{align}
We now expand the summands using the product rule and the chain rule \eqref{chainrule2}. Indeed, using the product rule to obtain the second equality below, and the chain rule to obtain the third, we have 
\begin{align} \label{eq:motivator}
 D^{ \mathbf{e_j} } \left[         \mathcal{E}_{F,G}(\mathcal{T} )  \right]  &:=
 D^{ \mathbf{e_j} } \left[     \left(  D_i^{\mu(v_0) } [F] \circ G \right) \prod_{ v \in V_1} D_{\tau(v)}^{\mu(v)} [G]  \right] \nonumber  \\ 
&=  D^{ \mathbf{e_j} } \left[   \left(  D_i^{\mu(v_0) } [F] \circ G \right)   \right] \prod_{v \in V_1} D_{\tau(v)}^{\mu(v)} [G] \nonumber \\  
&+    \sum_{v \in V_1}   \left(  D_i^{\mu(v_0) } [F] \circ G \right)   \prod_{w \in V_1} G_{\tau(v),\mu(v) + \ind_{v=w} \mathbf{e}_j } .
\end{align} 
Now note that by the chain rule, $D^{ \mathbf{e_j} } \left[   \left(  D_i^{\mu(v_0) } [F] \circ G \right)   \right]  = \sum_{ k = 1}^N \left(  D_i^{\mu(v_0) + \mathbf{e}_k } [F] \circ G \right) D_k^{\mathbf{e}_j} [G]$, so that 
\begin{align}  \label{symphony2}
 D^{ \mathbf{e_j} } \left[     \mathcal{E}_{F,G}(\mathcal{T} )  \right]  = \sum_{ k = 1}^N A_k(\mathcal{T}) + \sum_{ v \in V_1} B_v(\mathcal{T}),
\end{align}
where 
\begin{align*}
A_k(\mathcal{T}) := \left(  D_i^{\mu(v_0) + \mathbf{e}_k } [F] \circ G \right) D_k^{\mathbf{e}_j} [G] \prod_{v \in V_1} D_{\tau(v)}^{\mu(v)} [G]   
\end{align*}
and 
\begin{align*}
B_v(\mathcal{T})  :=  \left(  D_i^{\mu(v_0) } [F] \circ G \right)   \prod_{w \in V_1} G_{\tau(v),\mu(v) + \ind_{v=w} \mathbf{e}_j } .
\end{align*}
First, we observe that for each $k \in [N]$, $\mathcal{A}_k (\mathcal{T}) = \mathcal{E}_{F,G} \left( \mathcal{T}_k \right) $, where $\mathcal{T}_k$ is the tree in $\mathbb{F}_{i, \alpha + \mathbf{e}_j}(2)$ obtained from $\mathcal{T}$ by adding a vertex of type $k$ to $V_1$, and letting this vertex have the sole leaf $(j,\beta_j + 1)$ as a child. (Note that $\{ (j,\beta_j +1 ) \} = [ \beta + \mathbf{e}_j ] - [ \beta]$.)

Next, we note that $B_v (\mathcal{T}) = \mathcal{E}_{F,G} \left( \mathcal{T}_v \right)$ for each $v \in V_1$, where $\mathcal{T}_v$ is the tree in $\mathbb{F}_{i, \alpha + \mathbf{e}_j}(2)$ obtained from $\mathcal{T}$ by making the new leaf $(j,\beta_j+1)$ a child of $v$. 

We now note that every tree $\mathcal{T}'$ in $\mathbb{F}_{i,\beta+\mathbf{e}_j}(2)$ is obtained in this way by exactly one tree in $\mathbb{F}_{i,\beta}(2)$, that is, for each $\mathcal{T}' \in \mathbb{F}_{i,\beta + \mathbf{e}_j}(2)$ there is a unique $\mathcal{T}$ in $\mathbb{F}_{i,\beta}(2)$ such that either $\mathcal{T}' = \mathcal{T}_k$ or $\mathcal{T}' = \mathcal{T}_v$. 

Using \eqref{symphony1}, \eqref{symphony2} and this observation in the second equality below we obtain
\begin{align*}
D^{ \beta + \mathbf{e}_j }_i [ F \circ G ] &= \sum_{ \mathcal{T} \in \mathbb{F}_{i,\beta}(2) }  \left\{   \sum_{ k = 1}^N A_k(\mathcal{T}) + \sum_{ v \in V_1} B_v(\mathcal{T}) \right\}\\
&= \sum_{ \mathcal{T}' \in \mathbb{F}_{i,\beta + \mathbf{e}_j } (2) } \mathcal{E}_{F,G}(\mathcal{T}'),
\end{align*}
which proves the lemma.
 
\end{proof}

\subsection{The general case} \label{sec:fdb3}

\begin{proof}[Proof of Theorem \ref{thm:main}]
We now prove Theorem \ref{thm:main} by induction on $m$.

Now suppose the result holds for any $m$-fold composition of functions. We now show it holds for $(m+1)$-fold compositions. Indeed, let $F^{(1)} \circ \ldots \circ F^{(m+1)}$ be the composition of $m+1$ smooth functions from $\mathbb{R}^N \to \mathbb{R}^N$. By the special case $m =2$, with $F = F^{(1)}$ and $G = F^{(2)} \circ \ldots \circ F^{(m+1)}$, we have
\begin{align} \label{eq:orange}
D_i^\alpha [ F^{(1)} \circ \ldots \circ F^{(m+1) } ] (s) = \sum_{ \mathcal{T} \in \mathbb{F}_{i,\alpha} ( 2 ) } \left( D_i^{ \mu ( v_0 ) } [F] \circ G \right) \prod_{ v \in V_1} D_{ \tau(v)}^{\mu(v) } [G] .
\end{align}
We may now use the inductive hypothesis to compute $D_{\tau(v)}^{\mu(v)}[G]$. For each $v \in V_1$, by the inductive hypothesis we have 
\begin{align*}
D_{ \tau(v)}^{\mu(v) } [G] = \sum_{ \mathcal{T} \in \mathbb{F}_{\tau(v),\mu(v)} (2) } \mathcal{E}_{ G^*} \left( \mathcal{T} \right),
\end{align*}
where $G^*$ is the $m$-tuple $(F^{(2)},\ldots,F^{(m+1)}$. In particular, expanding \eqref{eq:orange} we have
\begin{align*}
D_i^\alpha [ F^{(1)} \circ \ldots \circ F^{(m+1) } ] (s) &= \sum_{ \mathcal{T} \in \mathbb{F}_{i,\alpha} ( 2 ) } \left( D_i^{ \mu ( v_0 ) } [F] \circ G \right) \prod_{ v \in V_1} \sum_{ \mathcal{T}_v \in \mathbb{F}_{ \mu(v),\tau(v) } (m) } \mathcal{E}_{ G^*} (\mathcal{T}_ v)\\ 
&= \sum_{ \mathcal{T} \in \mathbb{F}_{i,\alpha} ( 2 ) } \sum_{ \mathcal{T}_v \in \mathbb{F}_{ \mu(v), \tau(v) } (m)  : v \in V_1} \left( D_i^{ \mu ( v_0 ) } [F] \circ G \right) \prod_{ v \in V_1}  \mathcal{E}_{ G^*} (\mathcal{T}_ v).
\end{align*}
We now note that each term in the sum --- that is, each combination of $\mathcal{T} \in \mathbb{F}_{i,\alpha}(2)$ and a set of $m$-trees $(\mathcal{T}_v : v \in V_1)$ --- gives rise to a tree $\mathcal{T}'$ in $\mathbb{F}_{i,\alpha}(m+1)$ by `glueing' the tree $\mathcal{T}_v$ to the vertex $v$ in $V_1$, and that
\begin{align*}
\mathcal{E}_{ F^*} \left( \mathcal{T} \right) = \left( D_i^{ \mu ( v_0 ) } [F] \circ G \right) \prod_{ v \in V_1}  \mathcal{E}_{ G^*} (\mathcal{T}_ v) .
\end{align*}
We note that each $\mathcal{T}'$ in $\mathbb{F}_{i,\alpha}(m+1)$ is obtained uniquely in this way. It follows that  
\begin{align*}
D_i^\alpha [ F^{(1)} \circ \ldots \circ F^{(m+1) } ] (s) = \sum_{ \mathcal{T}' \in \mathbb{F}_{i,\alpha}(m+1) } \mathcal{E}_{F^*} ( \mathcal{T} ' ),
\end{align*}
proving the result.

\end{proof}

\section{Proof of the power series inversion formulas} \label{sec:lagrangeproof}

This section is dedicated towards proving the power series inversion formulas given in Theorem \ref{thm:inversion} and Theorem \ref{thm:general inversion}, as well as to proving Lemma \ref{lem:fern}, which was stated in our discussion of the Jacobian conjecture in Section \ref{sec:JC}.

We begin in the next section by making some first remarks about the ring of power series in $N$ commuting variables.
 
\subsection{The ring of formal power series}
Let $\mathbb{K}$ be a commutative ring, and consider the ring
\begin{align*}
R := \mathbb{K}[[X_1,\ldots,X_N]]
\end{align*}
of formal power series in $N$ commutative indeterminates with coefficients in $\mathbb{K}$. Namely, each element $f$ of $R$ is simply a collection $\left(f_{\alpha} :  \alpha \in \mathbb{Z}_{ \geq 0}^{N} \right)$ of elements of the underlying ring $\mathbb{K}$, though we think of $f$ as the formal expression
\begin{align*}
f(X) := \sum_{ \alpha \in \mathbb{Z}_{ \geq 0}^{N} } \frac{ f_{\alpha}}{\alpha!} X^\alpha.
\end{align*}
There are naturally defined notions of addition and multiplication of elements of $R$. Namely, $(f + g)_{\alpha} := f_\alpha + g_\alpha$, and by the Leibniz rule \eqref{eq:leib2}, the $\alpha^{\text{th}}$ coefficient of the product $fg$ is given by 
\begin{align*}
(fg)_\alpha =  \sum_{ S \sqcup T = [\alpha] } f_{ \# S} g_{ \# T}.
\end{align*}

Now define the set $R^N$ of objects of the form $F = (F_i)_{ i = 1,\ldots,N}$, where each $F_i$ is an element of $R$. In other words, $R^N$ is simply the set of tuples $F = \left( F_{i,\alpha} : i \in [N], \alpha \in \mathbb{Z}_{ \geq 0}^N \right)$. 

For each pair of elements $F,G$ in $R^N$, we would like to define the composition $F \circ G$ in $R^N$. To this end recall the special case $m=2$ of the Fa\`a di Bruno formula  Lemma \ref{thm:main2},  which states that 
\begin{align} \label{eq:scriabin2}
D^{\alpha}_i [F \circ G ] = \sum_{ \mathcal{T} \in \mathbb{F}_{i,\alpha} ( 2) } \mathcal{E}_{F,G}(\mathcal{T} )  \qquad \text{where    }\qquad \mathcal{E}_{F,G}(\mathcal{T} ) :=\left(  D_i^{\mu(v_0) } [F] \circ G \right) \prod_{ v \in V_1} D_{\tau(v)}^{\mu(v)} [G].
\end{align}
With \eqref{eq:scriabin2} in mind, we define the composition $F \circ G$ to be the element of $R^N$ whose coefficients are given by 
\begin{align} \label{eq:scriabin3}
(F \circ G)_{i, \alpha}  = \sum_{ \mathcal{T} \in \mathbb{F}_{i,\alpha} ( 2) } F_{i, \mu(v_0)} \prod_{ v \in V_1} D_{ \tau(v), \mu(v)}. 
\end{align}
Composition is clearly associative in that $(F \circ G) \circ H = F \circ (G \circ H)$, with the coefficients of multiple compositions afforded by the power series variant \eqref{eq:fdbpower} of Theorem \ref{thm:main}.

Define the element $I$ of $R^N$ by 
\begin{align*}
I_{i,\alpha} = 
\begin{cases}
1 \qquad &\text{if $\alpha = \mathbf{e}_i$},\\
0 \qquad &\text{otherwise}.
\end{cases}
\end{align*}
We note that $I$ is both a left and right identity for composition of functions in $R^N$, in that for every $F$ in $R^N$ we have $I \circ F = F \circ I = F$. 

We say an element $F$ of $R^N$ is invertible if there exists a $G$ in $R^N$ such that $F \circ G = G \circ F = I$. In Section \ref{sec:inversion proof} we prove Theorem \ref{thm:inversion}, which states that all elements $F$ of $R^N$ whose linear term is the identity matrix are invertible, and gives an explicit expression for the coefficients of the inverse. In the following section, Section \ref{sec:general inversion proof}, we prove Theorem \ref{thm:general inversion}, which states that an element $F$ of $R^N$ is invertible if and only if its linear term is invertible, and again gives an expression for the inverse.

\subsection{Inverses of power series with identity linear terms} \label{sec:inversion proof}

In this section we now give two proofs of Theorem \ref{thm:inversion} concerning the inverse of power series whose linear term is the identity. The first proof is short, and is simply a matter of verifying that the formula works. The second proof provides more insight, giving an inductive construction of the trees that helps motivate the formula.
\begin{proof}[Proof 1 of Theorem \ref{thm:inversion}. A verification proof]
Let $G$ be the power series whose coefficients are given by \eqref{eq:GDEF}. We now show directly that $F \circ G = I$. Indeed, separating \eqref{eq:scriabin3} into trees $\mathcal{T} \in \mathbb{F}_{i,\alpha}(2)$ in which the root has one child or more than one child, and using the fact that $F_{i,\mathbf{e}_j } = \delta_{i,j}$, we have 
\begin{align*}
(F \circ G)_{i,\alpha} = G_{i,\alpha } - \sum_{ \mathcal{T} \in \mathbb{F}_{i,\alpha} (2) : |\mu(v_0) | \geq 2 }H_{i,\mu(v_0) } \prod_{ v \in V_1 } G_{ \tau(v) , \mu (v) }.
\end{align*} 
We now claim that
\begin{align*}
\sum_{ \mathcal{T} \in \mathbb{F}_{i,\alpha} : |\mu(v_0) | \geq 2 }H_{i,\mu(v_0) } \prod_{ v \in V_1 } G_{ \tau(v) , \mu (v) } = G_{i,\alpha} := \sum_{ \mathcal{U} \in \mathbb{S}_{i,\alpha}} \mathcal{E}_H(\mathcal{U}).
\end{align*}
Indeed, let for each $\mathcal{T} \in \mathbb{F}_{i,\alpha}(2)$, and each $v$ in the first generation $V_1(\mathcal{T})$ of this tree, expand the term $G_{\tau(v),\mu(v)}$ in terms of a sum over trees $(\mathcal{T}_v)$. Each possible combination of of $\mathcal{T} \in \mathbb{F}_{i,\alpha}(2)$ and $\left(\mathcal{T}_v : v \in V_1(\mathcal{T})\right)$ creates a tree $\mathcal{U}$ in $\mathbb{S}_{i,\alpha}$ in such a way that $\mathcal{E}_H(\mathcal{U}) = H_{i,\mu(v_0)} \prod_{v \in V_1(\mathcal{T})} \mathcal{E}_H(\mathcal{T}_v )$. Moreover, every tree $\mathcal{U}$ in $\mathbb{S}_{i,\alpha}$ is created in this way by exactly one possible combination. 
\end{proof}

\begin{proof}[Proof 2 of Theorem \ref{thm:inversion}. A constructive proof]
The proof follows by an induction argument on the degree $|\alpha|$ of $\alpha$. By definition, if $F$ and $G$ are formal inverses of one another, we must have 
\begin{align} \label{cases 1}
(F \circ G)_{i ,\alpha } = 
\begin{cases}
1 \qquad &\text{if $\alpha = \mathbf{e}_i$},\\
0 \qquad &\text{otherwise}.
\end{cases}
\end{align}
On the other hand, by \eqref{eq:scriabin3} and the definition $(F \circ G)_{i , \alpha} := (F_i \circ G)_\alpha$, we have
\begin{align} \label{m2}
(F \circ G)_{i, \alpha}  = \sum_{ \mathcal{T} \in \mathbb{F}_{i,\alpha} ( 2) } F_{i, \mu(v_0)} \prod_{ v \in V_1} D_{ \tau(v), \mu(v)}. 
\end{align}
In particular, combining \eqref{cases 1} and \eqref{m2}, we obtain the following system of equations in $i$ and $\alpha$:
\begin{align} \label{m3}\sum_{ \mathcal{T} \in \mathbb{F}_{i,\alpha} ( 2) } F_{i, \mu(v_0)} \prod_{ v \in V_1} D_{ \tau(v), \mu(v)} =   
\begin{cases}
1 \qquad &\text{if $\alpha = \mathbf{e}_i$},\\
0 \qquad &\text{otherwise}.
\end{cases}
\end{align}
First we show that the linear term $J(G)(0)$ of the inverse $G$ is the identity map. This follows from noting that for $|\alpha| = 1$, the system of equations \eqref{m3} reads as saying 
\begin{align*}
\sum_{ k = 1}^N F_{i, \mathbf{e}_k } G_{ k, \mathbf{e}_j } = \delta_{i,j},
\end{align*}
or in other words, the matrix composition of $J(F)(0)$ and $J(G)(0)$ is the identity matrix. Since $J(F)(0)$ is the identity, so is $J(G)(0)$. 

Now when $|\alpha| = 2$, since every tree in $\mathbb{F}_{i,\alpha}(2)$ has two leaves, the generation $V_1$ either contains one or contains two vertices. In particular, setting $|\alpha| = 2$ in \eqref{m3} and using the fact that $F_{i,\mathbf{e}_j} = G_{i,\mathbf{e}_j } = \delta_{i,j}$, we obtain
\begin{align*}
F_{i,\alpha} + G_{i,\alpha} = 0 \qquad \text{for $|\alpha| = 2$},
\end{align*} 
which establishes \eqref{eq:GDEF} for $|\alpha| = 2$, since there is exactly one tree in $\mathbb{S}_{i,\alpha}$. The case $|\alpha| = 2$ will form a base case in our inductive proof.

Suppose now that $\beta \in \mathbb{Z}_{ \geq 0}^d$, and that the formula \eqref{eq:GDEF} holds for every $j$ in $[N]$ and every $\alpha \in \mathbb{Z}_{ \geq 0}^d$ satisfying $|\alpha| < |\beta|$. We now show it holds for $\beta$. Indeed, again using \eqref{m3}, and separating $F_i^\alpha(2)$ into those trees in which the root has one child or more children, and then using the fact that $F_{i,\mathbf{e}_j} = \delta_{i,j}$, we obtain
\begin{align*}
0 = G_{i,\alpha} + \sum_{ \mathcal{T} \in \mathbb{F}_{i,\alpha} (2) : |\mu(v_0)| \geq 2} F_{i, \mu(v_0)} \prod_{ v \in V_1} G_{ \tau(v), \mu(v)}.
\end{align*} 
Equivalently, by using the definition of $H$ we have 
\begin{align*}
 G_{i,\alpha} =  \sum_{ \mathcal{T} \in \mathbb{F}_{i,\alpha} (2) : |\mu(v_0)| \geq 2} H_{i, \mu(v_0)} \prod_{ v \in V_1} G_{ \tau(v), \mu(v)}.
\end{align*} 
Since $G_{i,\mathbf{e}_j} = \delta_{i,j}$, every vertex in $V_1$ either has more than one child, or they have a single child of the same type as themselves. With this in mind, define $\widetilde{\mathbb{F}}_{i,\alpha}(2)$ to be the subset of trees in $\mathbb{F}_{i,\alpha}(2)$ such that the root has two or more children, and such that every $v \in V_1$ with only one child, we have $\mu(v) = \mathbf{e}_{ \tau(v)}$. 
\begin{align} \label{eq:monkey}
 G_{i,\alpha} =  \sum_{ \mathcal{T} \in \widetilde{\mathbb{F}}_{i,\alpha} (2) } H_{i, \mu(v_0)} \prod_{ v \in V_1 : |\mu(v) | \geq 2 } G_{ \tau(v), \mu(v)}.
\end{align} 

Now take a tree $\mathcal{T}$ in $\mathbb{F}_{i,\alpha}(2)$. Since $|\mu(v_0)| \geq 2$, we have $|\mu(v)| < |\alpha|$ for every $v \in V_1$. In particular we may use the inductive hypothesis to expand  each $ G_{ \tau(v), \mu(v)} $ in terms of a tree. In particular, by using the inductive hypothesis in \eqref{eq:monkey} we have 
\begin{align*}
 G_{i,\alpha} &=  \sum_{ \mathcal{T} \in \widetilde{\mathbb{F}}_{i,\alpha} (2) } H_{i, \mu(v_0)} \prod_{ v \in V_1 : |\mu(v) | \geq 2 } \sum_{ \mathcal{T}_v  \in \mathbb{S}_{\tau(v), \mu(v)} }  \mathcal{E}_{H} (  \mathcal{T}_v) \\ 
&=  \sum_{ \mathcal{T} \in \widetilde{\mathbb{F}}_{i,\alpha} (2) } \sum_{ \left\{ \left( \mathcal{T}_v \right)_{v \in V_1} : \mathcal{T}_v  \in \mathbb{S}_{\tau(v), \mu(v)} : v \in V_1 \right\} }  H_{i, \mu(v_0)} \prod_{ v \in V_1 : |\mu(v) | \geq 2 }  \mathcal{E}_{H} (  \mathcal{T}_v).
\end{align*} 
Let $\mathcal{T} \in \widetilde{\mathbb{F}}_{i,\alpha}(2)$, and for each $v \in V_1$ let $\mathcal{T}_v  \in \mathbb{S}_{\tau(v), \mu(v)}$. Then it is easily seen that 
\begin{align*}
 H_{i, \mu(v_0)} \prod_{ v \in V_1 : |\mu(v) | \geq 2 }  \mathcal{E}_{H} (  \mathcal{T}_v) = \mathcal{E} \left( \mathcal{T}' \right),
\end{align*} 
where $\mathcal{T}'$ is the element of $\mathbb{S}_{i,\alpha}$ such that the subtree of $\mathcal{T}'$ of descendants of each vertex $v \in V_1$ is given by $\mathcal{T}_v$. Noting that every tree $\mathcal{T}'$ in $\mathbb{S}_{i,\alpha}$ is obtained uniquely in this way, we see that
\begin{align*}
G_{i,\alpha} = \sum_{ \mathcal{T}' \in \mathbb{S}_{i,\alpha} } \mathcal{E}_H \left( \mathcal{T}' \right),
\end{align*}
which proves the result.

\end{proof}

\subsection{Inverses of power series with non-identical linear terms}
\label{sec:general inversion proof}

In this section we prove the more general result, Theorem \ref{thm:general inversion}, concerning inversion of power series with non-identical linear terms. We will see that this result may be obtained as a corollary of Theorem \ref{thm:inversion}, where the linear term is assumed to be the identity matrix.

\begin{proof}[Proof of Theorem \ref{thm:general inversion}]

 Namely, suppose we have a general element $F = \left( F_{i,\alpha} : i \in [N], \alpha \in \mathbb{Z}_{ \geq 0}^N \right) $ of $R^N$ whose $i^{\text{th}}$ component is given by 
\begin{align*}
F_i(X) = \sum_{ j = 1}^N P_{ i, \mathbf{e}_j } X_j + \sum_{ |\alpha| \geq 2} \frac{ F_{i,\alpha}}{ \alpha!} X^\alpha,
\end{align*}
and consider the formal inverse $G$ of $F$. We now note that if the matrix $\left( P_{i, \mathbf{e}_j } \right)$ is not invertible, then neither is $F$, since the linear term of a functional inverse $G$ is necessarily a matrix inverse of $F$. 

Now suppose that $\left( Q_{i, \mathbf{e}_j } \right)$ is a matrix inverse of $\left( P_{i, \mathbf{e}_j } \right)$, so that $Q$ may be thought of as an element of $R^N$ with $i^{\text{th}}$ component $Q_i(s) = \sum_{ j = 1}^N Q_{i, \mathbf{e}_j } X_j$. Let $F = P - H$, and define the power series
\begin{align*}
\widetilde{F} := F \circ Q.
\end{align*}  
It is easily verified that the $i^{\text{th}}$ component of $\widetilde{F}$ has the form 
\begin{align*}
\widetilde{F}_i(s) = X_i - \widetilde{H}_i(s),
\end{align*}
where $\widetilde{H} = H \circ Q$ contains only terms of degree two or higher. By Theorem \ref{thm:inversion} we have an expression for the coefficients of the inverse $\widetilde{G}$ of $\widetilde{F}$. Indeed, 
\begin{align*}
\widetilde{G}_i(X) = X_i + \sum_{ |\alpha| \geq 2} \frac{ \widetilde{G}_{i,\alpha}}{\alpha!  } X^\alpha,
\end{align*}
where
\begin{align} \label{Gguy}
\widetilde{G}_{i,\alpha} = \sum_{ \mathcal{T} \in \mathbb{S}_{i,\alpha} } \mathcal{E}_{ \widetilde{H} }( \mathcal{T} ).
\end{align}
Now it is easily verified that the element $G$ of $R^N$ given by
\begin{align*}
G = Q \circ \widetilde{G}
\end{align*} 
is the inverse of $F$, and we may use Fa\`a di Bruno's formula to compute the coefficients of $G$. Indeed, using \eqref{eq:scriabin2} and the fact that $Q_{i,\alpha}$ is zero whenever $\alpha$ has degree two or higher to obtain the first equality below, and \eqref{Gguy} to obtain the second, we have
\begin{align} \label{Sumguy}
G_{i,\alpha} &= \sum_{ j \in [N] } Q_{i, \mathbf{e}_j }\widetilde{G}_{j, \alpha} \nonumber \\
&= \sum_{ j \in [N] } Q_{i, \mathbf{e}_j } \sum_{ \mathcal{T} \in \mathbb{S}_{j,\alpha}} \mathcal{E}_{ \widetilde{H} } ( \mathcal{T} ) \nonumber \\
&= \sum_{ j \in [N] , \mathcal{T} \in \mathbb{S}_{j,\alpha}} Q_{i, \mathbf{e}_j }  \mathcal{E}_{ \widetilde{H} } ( \mathcal{T} ).
\end{align}

Now recall that $\mathbb{A}_{i,\alpha}$ is the set of alternating trees: those vertices with an even distance from the root have exactly one child, and vertices lying an odd distance from the root have two or more children.

Note that there is a natural projection $\mathsf{P} : \mathbb{A}_{i,\alpha} \to \cup_{ j \in [N] } \mathbb{S}_{j,\alpha}$ defined as follows. Let $\mathcal{A} $ be an alternating tree in $\mathbb{A}_{i,\alpha}$, and define an element $\mathcal{T}' = (V',E',\tau',\phi') = \mathsf{P} (\mathcal{A})$ of $\mathbb{S}_{i,\alpha}$ as follows:
\begin{itemize}
\item The vertex set of $\mathcal{T}'$ is given by 
\begin{align*}
V' := \bigcup_{k \geq 0} V_{2k+1}(\mathcal{A}),
\end{align*}
\item The edge set of $\mathcal{T}'$ is given by 
\begin{align*}
E' := \{ (v_0, v_1) \} \cup \{ (v,w) \in V' : \text{$v$ is a grandchild of $w$ in $\mathcal{A}$ or vice versa} \},
\end{align*}
where $(v_0,v_1)$ is the edge between the root and the unique vertex in $V_1$.
\item Finally, we let $\tau'$ and $\phi'$ be the restrictions of $\tau$ and $\phi$ to $V'$. (We note that every leaf of the alternating tree $\mathcal{T}$ is in an odd generation, and hence also an element of $V'$.) 
\end{itemize}

Given $\mathcal{T}^*$ in $\mathbb{S}_{i,\alpha}$, define $\mathsf{P}^{-1} (\mathcal{T}^* )$ to be the subset of $\mathbb{A}_{i,\alpha}$ consisting of alternating trees $\mathcal{A}$ such that $\mathsf{P} (\mathcal{A} ) = \mathcal{T}^*$. The result now follows from the simple observation that  
\begin{align} \label{eq:rachmaninov}
 Q_{i, \mathbf{e}_j }  \mathcal{E}_{ H \circ Q } ( \mathcal{T} ) = \sum_{ \mathcal{A} \in \mathsf{P}^{-1}(\mathcal{T} ) } \mathcal{E}^{\mathsf{alt}}_{H,Q} (\mathcal{A}).
\end{align}
Indeed, plugging \eqref{eq:rachmaninov} into \eqref{Sumguy} we obtain
\begin{align*}
G_{i, \alpha} = \sum_{ j \in [N] , \mathcal{T} \in \mathbb{S}_{j,\alpha}}   \sum_{ \mathcal{A} \in \mathsf{P}^{-1}(\mathcal{T} ) } \mathcal{E}^{\mathsf{alt}}_{H,Q} (\mathcal{A}).
\end{align*}
The result follows by noting that we may write $\mathbb{A}_{i,\alpha}$ as a disjoint union
\begin{align*}
\mathbb{A}_{i,\alpha} = \bigcup_{ j \in [N] } \bigcup_{ \mathcal{T} \in \mathbb{S}_{j,\alpha}} \mathsf{P}^{-1}( \mathcal{T}).
\end{align*}

\end{proof}

\subsection{Proof of Lemma \ref{lem:fern}} \label{sec:JC proofs}

In this section we prove the \emph{fern lemma}, Lemma \ref{lem:fern}, which was used in our combinatorial restatement of the Jacobian conjecture, Conjecture \ref{JC3}. 

\begin{proof}[Proof of Lemma \ref{lem:fern}]
Let $H$ be a polynomial mapping containing only terms of degree two or higher. The $(i,j)^{\text{th}}$ entry of the matrix $J(H)^m$ may be written 
\begin{align*}
G_{i,j} ( X) := \sum_{ k_0, \ldots, k_{m} \in [N] } \ind \{ k_0 = i, k_m = j \} \prod_{ l =1 }^m D_{ k_{l-1} }^{ \mathbf{e}_{k_l} } [H] (X).
\end{align*}
Clearly, since $H$ contains only degree two terms of higher, $G_{i,j}(X)$ contains only terms of degree $m$ or higher, and hence for $G_{i,j}(0) = 0$ whenever $m \geq 1$. 

The property that $J(H)^m$ is equal to the zero matrix on $\mathbb{C}^N$ is equivalent to $G_{i,j}(X) = 0$ for all $1 \leq i,j \leq N$ and all $X \in \mathbb{C}^N$. Equivalently, all derivatives of $G_{i,j}$ are also equal to zero, so that for every $\alpha \in \mathbb{Z}_{ \geq 0}^d$,
\begin{align} \label{debussy1}
0 = D^\alpha [ G_{i,j} ](X) =  \sum_{ k_0, \ldots, k_{m} \in [N] } \ind \{ k_0 = i, k_m = j \} D^\alpha \left[ \prod_{ l =1 }^m D_{ k_{l-1} }^{ \mathbf{e}_{k_l} } [H] \right] (X) .
\end{align}
Using the multivariate Leibniz rule \eqref{eq:leib2}, we see that
\begin{align} \label{debussy2}
D^\alpha \left[ \prod_{ l =1 }^m D_{ k_{l-1} }^{k_l } [H] \right] = \sum_{ S_1 \sqcup \ldots \sqcup S_m = [\alpha] } \prod_{ l =1 }^m D_{k_{l-1}}^{\mathbf{e}_{k_l} + \# S_l } [ H]. 
\end{align}
We now note that every tuple of numbers $(k_0,\ldots,k_m) \in [N]$ with $k_0 = i$ and $k_m = j$, together with any sequence of sets $(S_1,\ldots,S_m)$, gives rise to a fern $\mathcal{T}$ in $\mathrm{Fern}_{i,\alpha,j}(m)$ by letting the leaves with labels in $S_i$ be children of $v_{i-1}$, and letting the type of $v_i$ be $k_i$. Moreover, the energy of such a fern is given by 
\begin{align} \label{debussy3}
\mathcal{E}_H \left( \mathcal{T} \right) =  \prod_{ l =1 }^m D_{k_{l-1}}^{\mathbf{e}_{k_l} + \# S_l } [ H].
\end{align}
In particular, by plugging \eqref{debussy2} into \eqref{debussy1}, and using \eqref{debussy3}, we see that \eqref{debussy1} reads as saying 
\begin{align*}
\sum_{ \mathcal{T} \in \mathrm{Fern}_{i,\alpha,j}( m ) } \mathcal{E}_H \left( \mathcal{T} \right) = 0. 
\end{align*}
\end{proof}

\section{The non-commutative case} \label{sec:free proofs}
In this section, we shall present the non-commutative/free versions of our main results and the reader should recall the notation introduced in Section \ref{sec:free}. In Section \ref{sec:free fdb proof} we prove the non-commutative generalisation of Fa\`a di Bruno's formula, Theorem \ref{thm:free fdb}. In Section \ref{sec:free inversion proof} we prove the non-commutative inversion formulas, Theorems \ref{thm:free inversion} and \ref{thm:free inversion general}. 

\subsection{Proof of the free Fa\`a di Bruno formula}
 \label{sec:free fdb proof}

In analogy with proving the commutative version of Fa\`a di Bruno's formula, in working towards a proof of Theorem \ref{thm:free fdb}, we begin by considering the case $m =2$. 

The case $m=2$ of Theorem \ref{thm:free fdb} reads as saying that whenever $F,G$ are power series in $\fR^N_0$ with zero constant term, then the coefficient of $X_{\kappa_1} \ldots X_{\kappa_k}$ is given by 
\begin{align} \label{eq:free1}
(F \circ G)_{ i , \kappa } = \sum_{ \freeT \in \fF_{i,\kappa}(2) } F_{ i, \fmu (v_0 )} \prod_{ v \in V_1} G_{ \tau(v), \fmu(v) }.
\end{align}
We now prove this equation. 

\begin{proof}[Proof of equation \eqref{eq:free1}]
Note that if $F$ and $F'$ are elements of $\fR^N$, then
\begin{align*}
(F + F') \circ G = F \circ G + F' \circ G. 
\end{align*}
In particular, $(F \circ G)_{i, \kappa}$ depends linearly on $F$. Since the right-hand side of \eqref{eq:free1} also clearly depends linearly on the coefficients of $F$, it is sufficient to prove \eqref{eq:free1} for $F_i$ of the form
\begin{align} \label{eq:Fform}
F_i (X_1,\ldots,X_N ) = X_{\gamma_1} \ldots X_{\gamma_j}.
\end{align}
For $F$ of the form \eqref{eq:Fform}, $F_i \circ G$ is given by 
\begin{align*}
F_i \circ G(X_1,\ldots,X_N) = G_{\gamma_1}( X_1,\ldots,X_N) \ldots G_{ \gamma_j} (X_1,\ldots,X_N) .
\end{align*}
Clearly then the coefficient of $X_{\kappa_1} \ldots X_{ \kappa_k }$ in $(F_i \circ G)$ is given by 
\begin{align} \label{eq:kappasum}
(F \circ G)_{i,\kappa} = \sum_{ 1 \leq l_1 < l_2 < \ldots < l_{j-1} <  k } G_{\gamma_1, (\kappa_1,\ldots,\kappa_{l_1}) } G_{\gamma_2, (\kappa_{l_1 + 1}, \ldots, \kappa_{l_2} )} \ldots G_{ \gamma_j, (\kappa_{l_{j-1} + 1} , \ldots, \kappa_k ) }. 
\end{align}

\begin{figure}[ht!]
  \centering   
 \begin{forest}
[$i$, name=L1,
[$\gamma_1$,
[$\kappa_1$, name=A]
[$\ldots$]
[$\kappa_{l_1}$]
]
[$\gamma_2$,
[$\kappa_{l_1 + 1}$]
[$\ldots$]
[$\ldots$]
[$\ldots$]
[$\kappa_{l_2}$]
]
[$\ldots$
[$\ldots$]
[$\ldots$]
]
[$\gamma_j$
[$\small{\kappa_{l_{j-1}+1}}$]
[$\ldots$]
[$\ldots$]
[$\kappa_k$]
]
]
]
]
\node (a) [left=of L1] {($\mathsf{Planar}$)};
\end{forest}
  \caption{The coefficient of $X_{\kappa_1} \ldots X_{\kappa_k}$ in the $i^{\text{th}}$ component of $(F \circ G)$ may be given in terms of a sum over planar trees with two generations.}\label{fig:planar partitions}
\end{figure}
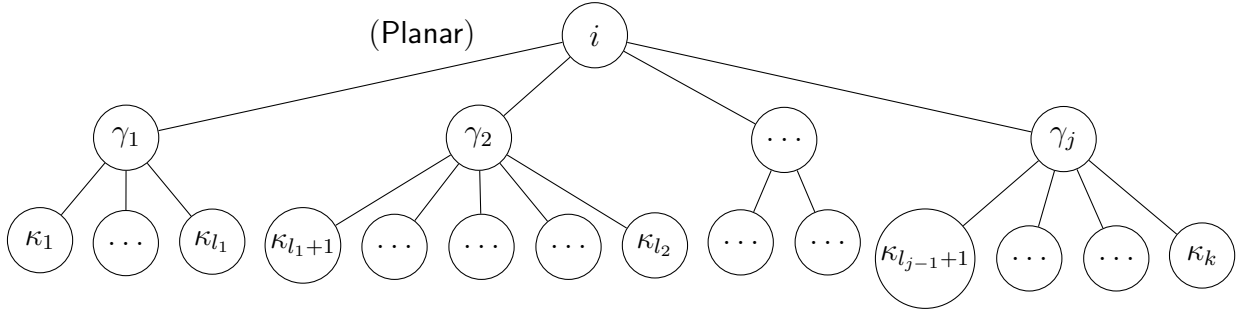

Now there is a bijection between every tuple $1 \leq l_1 < l_2 < \ldots < l_{j-1} < k$ the subset of $\fF_{i,\kappa}(2)$ consisting of the planar trees with $j$ vertices $\{v_1,\ldots,v_j\}$ in generation $1$ (listed in planar order) by letting the leaves with labels  $\{ l_{j-1}+1,\ldots,l_j \}$ be descended from $v_j$. Write $\mathcal{T} ( l_1,\ldots,l_{j-1})$ for the unique tree with this property. It follows that \eqref{eq:kappasum} may be written
\begin{align} \label{eq:star}
(F_i \circ G)_\kappa = \sum_{ \freeT \in \fF_{i,\kappa}(2)_\gamma } \prod_{ v \in V_1} G_{ \tau(v), \fmu(v) },
\end{align} 
where $\fF_{i,\kappa}(2)_\gamma$ is the subset of $\fF_{i,\kappa}(2)$ consisting of those trees with vertices $\{v_1,\ldots,v_j\}$ in generation $1$ with types $\gamma_1,\ldots,\gamma_j$.

In other words, if $F_i$ has the form \eqref{eq:Fform}, then $F_{i,\gamma} = 1$ and $F_{i,\gamma'}  = 0$ for all tuples $\gamma' \neq \gamma$, so that we may write \eqref{eq:star} as 
\begin{align*}
(F \circ G)_{i, \kappa} = \sum_{ \freeT \in \fF_{i,\kappa}(2) } F_{i,  \fmu(v_0) } \prod_{ v \in V_1} G_{\tau(v), \fmu(v)},
\end{align*}
which amounts to \eqref{eq:free1} for $F_i(X_1,\ldots,X_k) = X_\gamma$, which by linearity, implies $\eqref{eq:free1}$ for general $F$. 

\end{proof}

We now prove Theorem \ref{thm:free fdb} by induction on $m$. Since the proof bears a great deal of similarity with our proof in the analogous commutative case, we will provide less detail here than we did in the commutative case.

\begin{proof}[Proof of Theorem \ref{thm:free fdb}]
We proceed by induction on $m$ by using the $m=2$ case proved above to prove the inductive step. 
Suppose we have a sequence $(F^{(1)},\ldots,F^{(m+1)}$) of elements of $\fR^N_0$. Then by letting 
$F = F^{(1)}$ and $G = (F^{(2)} \circ \ldots \circ F^{(m+1)})$ in \eqref{eq:free1}, we obtain
\begin{align} \label{eq:orange1}
(F^{(1)} \circ \ldots \circ F^{(m+1)})_{i ,\kappa}  = \sum_{ \mathcal{T} \in \fF_{i,\kappa} (2) } F^{(1)}_{ i, \fmu(v_0) } \prod_{ v \in V_1 } \left( F^{(2)} \circ \ldots \circ F^{(m+1)} \right)_{ \tau(v) , \fmu(v) } .
\end{align}
By the inductive hypothesis, we may expand each term
\begin{align} \label{eq:orange2}
 \left( F^{(2)} \circ \ldots \circ F^{(m+1)} \right)_{ \tau(v) , \fmu(v) } = \sum_{ \mathcal{T}_v \in \mathbb{F}_{\tau(v)}^{ \fmu(v) } (m) } \mathcal{E}_{ (F^{(2)},\ldots,F^{(m+1)} ) } (\freeT_v) .
\end{align}
Suppose now that $\freeT$ is a tree in $\fF_{i,\kappa}(2)$, such that the vertices in generation $1$ have types and free outdegrees $\{ (\tau(v) , \fmu(v) ) : v \in V_1 \}$. Given a collection of trees $\left( \freeT_v : v \in V_1 \right)$ such that each $\freeT_ v$ is an element of $\fF_{\tau(v),\fmu(v)} (m)$, we may create a tree $\mathcal{S} = \mathrm{Tree} \left(\freeT; \freeT_v : v \in V_1(\freeT ) \right)$ in $\fF_{i,\kappa}(m+1)$ by ``glueing" each $\mathcal{T}_v$ to the vertex $v$ in $\freeT$. Moreover, the $F^* = (F^{(1)},\ldots,F^{(m+1)})$-energy of the resulting tree is given by 
\begin{align*}
\mathcal{E}_{F^*}( \mathcal{S} ) = \prod_{ l=1}^{m+1} \prod_{v \in V_{l-1} } F^{(l)}_{\tau(v), \fmu(v)} = F_{i, \fmu(v_0) } \prod_{ v \in V_1 } \mathcal{E}_{ (F^{(2)},\ldots,F^{(m+1)} )} (\freeT_v).
\end{align*}
In particular, since every tree $\mathcal{S}$ has a unique decomposition into trees $\left( \freeT_v : v \in V_1(\mathcal{S}) \right)$ it follows that by plugging \eqref{eq:orange2} into \eqref{eq:orange1} that 
\begin{align*}
(F^{(1)} \circ \ldots \circ F^{(m+1)})_{i ,\kappa}  &= \sum_{ \mathcal{T} \in \fF_{i,\kappa} (2) } F^{(1)}_{ i, \fmu(v_0) } \prod_{ v \in V_1 } \left( F^{(2)} \circ \ldots \circ F^{(m+1)} \right)_{ \tau(v) , \fmu(v) } \\
&= \sum_{ \freeT \in \fF_{i,\kappa} (2) } F^{(1)}_{ i, \fmu(v_0) } \prod_{ v \in V_1 }\sum_{ \freeT_v \in \fF_{\tau(v),\fmu(v)} (m) } \mathcal{E}_{ (F^{(2)},\ldots,F^{(m+1)} )} ( \freeT_v )\\
&= \sum_{ \freeU \in \fF_{i,\kappa}(m+1) } \mathcal{E}_{ (F^{(1)},\ldots,F^{(m+1)} )} (\freeU ).
\end{align*}
This shows that the formula also holds for $m + 1$. 
\end{proof}

\subsection{Proof of the free inversion theorem}
 \label{sec:free inversion proof}

In this section we prove Theorem \ref{thm:free inversion}, which gives a formula for the compositional inverse of a formal power series in non-commutative indeterminates. Recall that in Section \ref{sec:inversion proof} we gave two proofs of the commutative analogue Theorem \ref{thm:inversion} of Theorem \ref{thm:free inversion} --- the first constructive and the latter a simple verification. 

In terms of proving Theorem \ref{thm:free inversion} here, we only give a verification proof, since a constructive argument may be intimated from paralleling the constructive argument in the commutative case.

\begin{proof}[Proof of Theorem \ref{thm:free inversion}]
Suppose we are in the context of Theorem \ref{thm:free inversion}. Namely, suppose $F$ and $G$ are elements of $\fR^N_0$ such that each component $F_i$ of $F$ has the form
\begin{align*}
F_{ i } (X_1,\ldots,X_N) = X_i - \sum_{ k \geq 2} \sum_{ \kappa \in [N]^k} H_{i, \kappa} X_\kappa,
\end{align*}
and each component of $G$ has the form
\begin{align*}
G_i ( X_1,\ldots,X_N)  = X_ i + \sum_{ k \geq 2} \sum_{ \kappa \in [N]^k } G_{i,\kappa} X_\kappa,
\end{align*}
where
\begin{align*}
G_{ i, \kappa} = \sum_{ \freeT \in \fS_{i,\kappa}} \mathcal{E}_H \left( \freeT\right).
\end{align*}
We now use the special case $m=2$ of the free Fa\`a di Bruno formula \eqref{eq:free1} to verify that both $F \circ G$ and $G \circ F$ are equal to the identity power series $I$. 

We begin with $F \circ G$. For $j \in [N]$, by letting $\kappa$ be the $1$-tuple $\kappa = (j)$ in \eqref{eq:free1}, we see that the coefficient of $X_j$ in the $i^{\text{th}}$ component of $(F \circ G)$ is given by 
\begin{align*}
(F \circ G)_{ i , j } = \sum_{ k \in [N] } F_{i , k } G_{k , j } = \ind_{ i =j}.
\end{align*}
We now show that whenever $k \geq 2$, for any $\kappa \in [N]^k$, the coefficient of $X_\kappa$ in the $i^{\text{th}}$ component of $(F \circ G)$ is zero. Indeed, by \eqref{eq:free1}
\begin{align} \label{giraffe}
(F \circ G)_{ i , \kappa} = \sum_{ \freeT \in \fF_{i,\kappa}(2) } F_{ i , \fmu(v_0) } \prod_{ v \in V_1 } G_{ \tau(v) , \fmu(v) }  = G_{i,\kappa } -  \sum_{ \freeT \in \fF_{i,\kappa}(2) : |V_1| > 2 } H_{i, \fmu(v_0)} \prod_{ v \in V_1 }  G_{ \tau(v) , \fmu(v) }.
\end{align}
We now note that each $G_{ \tau(v), \fmu(v) }$ may be expended in terms of a sum over trees in $\fS_{ \tau(v), \fmu(v) }$. In particular, 
\begin{align} \label{elephant}
\sum_{ \freeT \in \fF_{i,\kappa}(2) : |V_1| > 2 } H_{i, \fmu(v_0)} \prod_{ v \in V_1 }  G_{ \tau(v) , \fmu(v) } &= \sum_{ \freeT \in \fF_{i,\kappa}(2) : |V_1| > 2 }  \sum_{ \left( \freeT_v \right)_{v \in V_1} ,\freeT_v \in \fS_{ \tau(v), \fmu(v) } } H_{i, \fmu(v_0)} \prod_{ v \in V_1 }  \mathcal{E}_H \left(\freeT_v \right).
\end{align}
However, for each term in the summand on the right-hand side, we may write 
\begin{align*}
\mathcal{E} \left( \freeT^* \right) = H_{i, \fmu(v_0)} \prod_{ v \in V_1 }  \mathcal{E}_H \left( \freeT_v \right),
\end{align*}
where $\freeT^*$ is the tree in $\fS_{i, \kappa}$  obtained by ``glueing" each $\freeT_v$ to $v$. Moreover, each tree $\freeT^*$ in $\fS_{i,\kappa}$ is obtained uniquely in this from a tree $\freeT \in \fF_{i,\kappa}(2)$, and a collection $(\freeT_v)$ of trees such that each $\freeT_v$ is contained in $\fS_{\tau(v), \fmu(v)}$. 

In particular, \eqref{elephant} reads as saying
\begin{align}
\sum_{ \freeT \in \fF_{i,\kappa}(2) : |V_1| > 2 } H_{i, \fmu(v_0)} \prod_{ v \in V_1 }  G_{ \tau(v) , \fmu(v) } = \sum_{ \freeT^* \in \fS_{i,\kappa} } \mathcal{E}_H \left( \freeT^*\right),
\end{align}
so that by \eqref{giraffe}, $(F\circ G)_{i, \kappa} = 0$. The proof that $G \circ F$ is equal to the identity is similar, and left to the reader. 
\end{proof}

We now turn to proving Theorem \ref{thm:free inversion general}, concerning inverses of power series with non-identical linear terms in $N$ free indeterminates. Since the proof of Theorem \ref{thm:free inversion general} is extremely similar to our proof of its commutative counterpart, Theorem \ref{thm:general inversion}, we shall only sketch the main details. 

\begin{proof}[Proof of Theorem \ref{thm:free inversion general}]
Let $F$ be a free power series in $\fR^N$ with components of the form
\begin{align*}
F_i(X_1,\ldots,X_N) = \sum_{ j \in [N] } P_{i,j} X_j  - \sum_{ k \geq 2} \sum_{ \kappa \in [N]^k } H_{i, \kappa} X_\kappa, \qquad i \in [N].
\end{align*}
Now consider the element $\widetilde{F}$ of $\fR^N$ defined by setting $\widetilde{F} = F \circ Q$. It is easily verified using \eqref{eq:free1} that $\widetilde{F}$ has components of the form 
\begin{align*}
\widetilde{F}_i(X_1,\ldots,X_N) = X_i  - \sum_{ k \geq 2} \sum_{ \kappa \in [N]^k } \widetilde{H}_{i, \kappa} X_\kappa, \qquad i \in [N].
\end{align*}
where $\widetilde{H} = H \circ Q$. In particular, by Theorem \ref{thm:free inversion} the inverse $\widetilde{G}$ exists and has components of the form
\begin{align*}
\widetilde{G}_i(X_1,\ldots,X_N )= X_i + \sum_{ k \geq 2} \sum_{ \kappa \in [N]^k } G_{i,\kappa } X_\kappa,
\end{align*}
where
\begin{align*}
\widetilde{G}_{i, \kappa} = \sum_{ \freeT \in \fS_{i,\kappa}} \mathcal{E}_{\widetilde{H}}(\freeT ).
\end{align*}
It is easily verified now that $G := Q \circ \widetilde{F}$ is inverse to $F$. Moreover, in a planar expansion identical to its non-planar counterpart in the proof of Theorem \ref{thm:general inversion}, we see that the coefficients of $G$ are given by 
\begin{align*}
G_{i,\kappa} = \sum_{ \freeT \in \fA_{i,\kappa} } \mathcal{E}^{\mathsf{alt}}_{H,Q}( \freeT ),
\end{align*}
as required.
\end{proof}

\subsection*{Acknowledgments}  
The authors are extremely grateful to an anonymous referee whose suggestions have greatly improved this article. We would also like to thank Michael Anshelevich for directing us towards several useful references.

SJ and JP have been supported by the Austrian Science Fund (FWF) Project P32405 ``Asymptotic Geometric Analysis and Applications'' of which JP is principal investigator. JP has also been supported by a \textit{Visiting International Professor Fellowship} from the Ruhr University Bochum and its Research School PLUS.

Finally, we gratefully acknowledge the support of the Oberwolfach Research Institute for Mathematics, where several discussion were held during the workshop ``New Perspectives and Computational Challenges in High Dimensions'' (Workshop ID 2006b).

\end{document}